\documentclass{article}
\usepackage[utf8]{inputenc}
\usepackage{tikz}
\usepackage{hyperref,verbatim}
\usepackage{enumitem,bm}
\usepackage{todonotes}

\usepackage{pgf,tikz,tkz-graph,subcaption}
\usetikzlibrary{arrows,shapes}
\usetikzlibrary{decorations.pathreplacing}
\usepackage{fullpage}
\usetikzlibrary{decorations.pathmorphing}

\newcommand{\red}[1]{{\color{red} #1}}

\definecolor{purple}{RGB}{208,134,255}
\definecolor{blue}{RGB}{10,120,253}
\definecolor{red}{RGB}{254,144,33}
\definecolor{green}{RGB}{126,198,52}
\definecolor{orange}{RGB}{244,154,33}

\usepackage{amsthm,amsmath,amssymb}
\usepackage{url}
\usepackage[capitalise]{cleveref}
\usepackage{cite}
\def\aftermath{\par\vspace{-\belowdisplayskip}\vspace{-\parskip}\vspace{-\baselineskip}}

\newtheorem{conj}{Conjecture}

\newtheorem{defi}[conj]{Definition}
\newtheorem{theorem}[conj]{Theorem}
\newtheorem{lem}[conj]{Lemma}
\newtheorem{ques}[conj]{Question}
\newtheorem{cor}[conj]{Corollary}

\newtheorem{prop}[conj]{Proposition}
\newtheorem{ex}[conj]{Example}
\newtheorem{remark}[conj]{Remark}
\newtheorem{claim}[conj]{Claim}
\newtheorem*{ORL}{Optimal Renaming Lemma}
\newtheorem*{proptwo}{Proposition 2}

\newcommand*{\myproofname}{Proof}
\newenvironment{claimproof}[1][\myproofname]{\begin{proof}[#1]}{\end{proof}}

\DeclareMathOperator{\rad}{rad}

\DeclareMathOperator{\diam}{diam}

\renewcommand{\a}{\alpha}
\renewcommand{\b}{\beta}

\def\C{\mathcal{C}}

\newcommand{\abs}[1]{\left|#1\right|}

\newcommand*{\ceilfrac}[2]{\mathopen{}\left\lceil\frac{#1}{#2}\right\rceil\mathclose{}}
\newcommand*{\floorfrac}[2]{\mathopen{}\left\lfloor\frac{#1}{#2}\right\rfloor\mathclose{}}

\title{Sharp Bounds on Lengths of Linear
Recolouring Sequences}

\author{
	Stijn Cambie%
 \thanks{Department of Computer Science, KU Leuven Campus Kulak-Kortrijk, 8500 Kortrijk, Belgium. 
 Supported by a FWO grant with grant number 1225224N. Email: \protect\href{mailto:stijn.cambie@hotmail.com}{\protect\nolinkurl{stijn.cambie@hotmail.com}}.}	
	\and
	Wouter Cames van Batenburg%
	\thanks{
 Delft Institute of Applied Mathematics,
 Delft University of Technology, Netherlands.
		Email: \protect\href{mailto:w.p.s.camesvanbatenburg@gmail.com}{\protect\nolinkurl{w.p.s.camesvanbatenburg@gmail.com}}.}
 \and 
Daniel W. Cranston\thanks{Virginia Commonwealth University, Dept.~of
Computer Science, Richmond, Virginia, USA;
Email: \protect\href{mailto:dcranston@vcu.edu}{\protect\nolinkurl{dcranston@vcu.edu}}.}	
}

\date{}

\begin{document}

\maketitle

\begin{abstract}
A recolouring sequence, between $k$-colourings $\a$ and $\b$ of a graph $G$, transforms $\a$ into $\b$ by recolouring one vertex at a time, such that after each recolouring step we again have a proper $k$-colouring of $G$. The diameter of the $k$-recolouring graph, $\diam \C_k(G)$, is the maximum over all pairs $\a$ and $\b$ of the minimum length of a recolouring sequence from $\a$ to $\b$. Much previous work has focused on determining the asymptotics of $\diam \C_k(G)$: Is it $\Theta(|G|)$? Is it $\Theta(|G|^2)$? Or even larger?
Here we focus on graphs for which $\diam \C_k(G)=\Theta(|G|)$, and seek to determine more precisely the multiplicative constant implicit in the $\Theta()$. In particular, for each $k\ge 3$, for all positive integers $p$ and $q$ we exactly determine $\diam \C_k(K_{p,q})$, up to a small additive constant.
We also sharpen a recolouring lemma that has been used in multiple papers, proving an optimal version. This improves the multiplicative constant in various prior results.
Finally, we investigate plausible relationships between similar reconfiguration graphs.
\end{abstract}

\section{Introduction}
The \emph{$k$-recolouring graph}, $\C_k(G)$, has as its vertices the $k$-colourings of $G$; two vertices are adjacent precisely when those $k$-colourings differ in the colour of a single vertex of $G$. A walk in $\C_k(G)$ between two $k$-colourings $\a$ and $\b$ is called a \emph{recolouring sequence} from $\a$ to $\b$. It is well-known~\cite{jerrum}, and easy to prove by induction, that $\C_k(G)$ is connected when $k\ge \Delta(G)+2$. In fact, in this case the diameter of $\C_k(G)$ is linear in $|G|$; 
specifically~\cite[Thm.~1]{CCC24}, we have $\diam\C_k(G)\le 2|G|$. Much previous research~\cite{BB18,BC25,BCH24,BCM24,BP15,merkel,BFHR22,DF20,DF21,cranston22} has focused on determining sufficient conditions on integers $k$ and graph classes $\mathcal{G}$ to imply that $\diam\C_k(G)=\Theta(|G|)$ for all $G\in\mathcal{G}$.
In this paper we revisit some of these graphs, and prove sharper upper bounds on $\diam \C_k(G)$.

In Section~\ref{opt-sec}, we prove the Optimal Renaming Lemma, which is a sharper version of a lemma used in numerous papers.
As a consequence, we improve the upper bounds in these previous works.

In Section~\ref{Kpq-sec}, 
we consider the complete bipartite graph $K_{p,q}$; for each number $k$ of colours we determine,
up to a small additive constant, $\diam\C_k(K_{p,q})$. We show, perhaps surprisingly, 
that $\diam \C_k(K_{p,q})$ moves through 3 distinct phases, as the ratio $q/p$ decreases.
Namely, we prove our main result, stated below.

\begin{theorem}\label{thr:main_diamCk(Kpq)}
Fix positive integers $p,q,k$. If $p \le q$ and $k\ge 3$, then 
$$\diam\C_k(K_{p,q})=
\left.
 \begin{cases}
 2p+q+\floorfrac{q-p}{k-1} & \text{if } q \ge kp \\ \\
 3p+q+\floorfrac{q-kp}{\floorfrac{k^2}{4} } -g_1(k,p,q) & \text{if } \ceilfrac {k}{2} p \le q \le kp,\\
 \\
 2p+q+\floorfrac{q-p}{\ceilfrac {k}{2} } -g_2(k,p,q)& \text{if } q \le \ceilfrac {k}{2} p, 
 \end{cases}
 \right.$$
 where $0 \le g_1(k,p,q) \le \floorfrac{k^2}{4}$ and $0 \le g_2(k,p,q) \le \floorfrac{k}{2}.$ 
\end{theorem}

Finally, in Section~\ref{surprise-sec} we consider list-colouring. For each list-assignment $L$ we define $\C_L(G)$, analogous
to $\C_k(G)$. Naturally, we look for relationships between $\diam\C_L(G)$ and $\diam \C_k(G)$, particularly when $|L(v)|=k$ for 
all vertices $v$. We mainly construct examples disproving many statements we might hope were true.
We also pose as open problems some statements that we have neither been able to prove or disprove, regarding relationships between $\diam \C_k(G)$ and $\diam \C_{k+1}(G)$, or between $\rad \C_k(G)$ and $\rad \C_{k+1}(G)$.

Most of our notation is standard.
We note that we write $[i]$ for $\{1,\ldots,i\}$ and we write $[i,j]$ for $\{i,\ldots,j\}$.

\subsection{Intuition behind \texorpdfstring{Theorem~\ref{thr:main_diamCk(Kpq)}}{Theorem 1}}
To better understand Theorem~\ref{thr:main_diamCk(Kpq)}, it is instructive to compare to each other the upper bounds for its $3$ regimes. 
When $k=3$, these upper bounds
for all regimes are identical; furthermore, in that case the functions $g_1$ and $g_2$ are identically zero. We state this formally below
and prove it just before Subsection~\ref{upper:ssec}.
\begin{prop}
\label{k=3:prop}
 For every complete bipartite graph $K_{p,q}$, we have $\diam \C_3(K_{p,q})= \floorfrac{3(p+q)}{2}.$
\end{prop}

In~\cref{fig:graph_of_functions} below, we plot the 3 upper bounds in Theorem~\ref{thr:main_diamCk(Kpq)} as functions in $q$, for $k=4$ and $p$ fixed.  But where do these bounds come from?
A natural guess for colourings $\a$ and $\b$ that are far apart is to have $\a$ use $a$ of the colours, equally, on part $U$ and the remaining $k-a$ of the colours, equally, on part $V$; and $\b$ swaps the parts of all colours, relative to $\a$, again using the colours on each part as equally as possible.

Formalising this intuition does indeed give a pair $\a,\b$ that is far apart. But we still must specify $a$, the number of colours used by $\a$ on $U$. When $q$ is large relative to $p$, namely $q\ge kp$, the worst case (largest distance) arises with $a=1$; and when $q<kp$, the worst case arises with $a=\floorfrac{k}{2}$.
For the pair $\a,\b$ of colourings arising with $a=1$, the number of recolouring steps needed is shown in~\cref{fig:graph_of_functions} in blue (we call this function $\mathcal{B}$).
For the pair $\a,\b$ of colourings arising with $a=\floorfrac{k}2$, we have two candidate recolouring strategies, shown in~\cref{fig:graph_of_functions} in orange and green (functions $\mathcal{O}$ and $\mathcal{G}$). The first of these ($\mathcal{O}$) is optimal when $q\le\ceilfrac{k}2p$, and the second ($\mathcal{G}$) is optimal when $\ceilfrac{k}2p\le q\le kp$.
So our upper bound for the diameter is $\max\{\mathcal{B}, \min\{\mathcal{O},\mathcal{G}\} \}.$

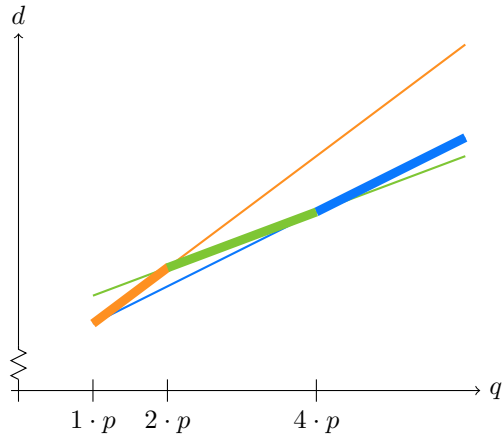
\begin{figure}[ht]
 \centering
\begin{tikzpicture}[scale=0.99,yscale=1.5]
 \draw[->] (-0.1, 0) -- (6.2, 0) node[right] {$q$};
 \draw[->] (0., 0.4) -- (0, 3.2) node[above] {$d$};
 \draw[line width=.15mm,decorate,decoration={zigzag,segment length=2mm,amplitude=1mm}] (0,0.1) -- (0,0.4);
 \draw (0,-0.1)--(0,0.1);w
 \draw[domain=1:6, thick, smooth, variable=\x, blue] plot ({\x}, {1/3*\x+1-1/3-2/5});
\draw[domain=1:6, thick, smooth, variable=\x, green] plot ({\x}, {1/4*\x+1-2/5});
 \draw[domain=4:6, line width=1.2mm, variable=\x, blue] plot ({\x}, {1/3*\x+1-1/3-2/5});
 \draw[domain=1:6, smooth, thick, variable=\x, red] plot ({\x}, {1/2*\x+1-1/2-2/5});
 \draw[domain=1:2.02, line width = 1.2mm, variable=\x, red] plot ({\x}, {1/2*\x+1-1/2-2/5});
\draw[domain=2:4, line width = 1.2mm, variable=\x, green] plot ({\x}, {1/4*\x+1-2/5});
 \foreach \y in {1,2,4}{
 \draw (\y,0.1)--(\y,-0.1) node[below]{$\y \cdot p$};
 }
\end{tikzpicture}
\captionsetup{width=.85\textwidth}
 \caption{The three upper bounds in Theorem~\ref{thr:main_diamCk(Kpq)} for $k=4$ and $p \le q \le 6p$. The bold line segments indicate the value of $\diam\C_k(K_{p,q})$, up to the error terms $g_1(k,p,q)$ and $g_2(k,p,q)$.}
 \label{fig:graph_of_functions}
\end{figure}

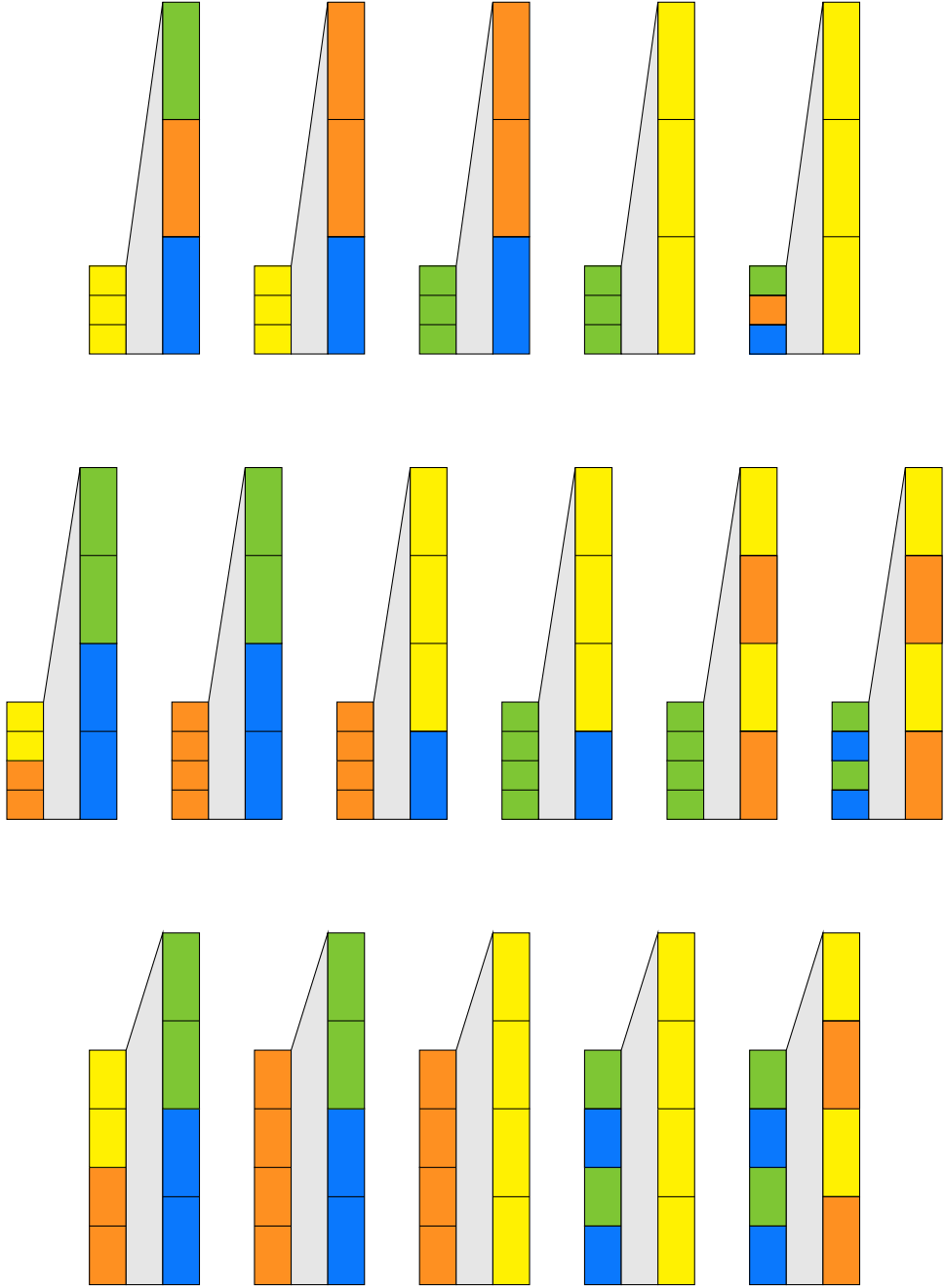
\begin{figure}[!ht]
\def\myspace{2.25cm}
\centering
\begin{centering}
\begin{tikzpicture}[yscale=.8]
 \draw [fill=yellow] (0,0) rectangle (0.5,1.5);
 \draw [fill=black!10!white] (0.5,0)--(1,0)--(1,6)--(0.5,1.5) -- cycle;
 \draw [fill=blue] (1,0) rectangle (1.5,2);
 \draw [fill=red] (1,4) rectangle (1.5,2);
 \draw [fill=green] (1,4) rectangle (1.5,6);

\draw (0,0.5)--(0.5,0.5);
\draw (0,1)--(0.5,1);
\draw (1,2)--(1.5,2);
\draw (1,4)--(1.5,4);
 
\begin{scope}[xshift=\myspace]
 \draw [fill=yellow] (0,0) rectangle (0.5,1.5);
 \draw [fill=black!10!white] (0.5,0)--(1,0)--(1,6)--(0.5,1.5) -- cycle;
 \draw [fill=blue] (1,0) rectangle (1.5,2);
 \draw [fill=red] (1,6) rectangle (1.5,2);
 \draw (0,0.5)--(0.5,0.5);
\draw (0,1)--(0.5,1);
\draw (1,2)--(1.5,2);
\draw (1,4)--(1.5,4);
\end{scope}
\begin{scope}[xshift=2*\myspace]
 \draw [fill=green] (0,0) rectangle (0.5,1.5);

 \draw [fill=blue] (1,0) rectangle (1.5,2);
 \draw [fill=red] (1,6) rectangle (1.5,2);
 \draw [fill=black!10!white] (0.5,0)--(1,0)--(1,6)--(0.5,1.5) -- cycle;
 \draw (0,0.5)--(0.5,0.5);
\draw (0,1)--(0.5,1);
\draw (1,2)--(1.5,2);
\draw (1,4)--(1.5,4);
\end{scope}
\begin{scope}[xshift=3*\myspace]
 \draw [fill=green] (0,0) rectangle (0.5,1.5);

 \draw [fill=yellow] (1,0) rectangle (1.5,6);

 \draw [fill=black!10!white] (0.5,0)--(1,0)--(1,6)--(0.5,1.5) -- cycle;
 \draw (0,0.5)--(0.5,0.5);
\draw (0,1)--(0.5,1);
\draw (1,2)--(1.5,2);
\draw (1,4)--(1.5,4);
\end{scope}
\begin{scope}[xshift=4*\myspace]
 \draw [fill=yellow] (1,0) rectangle (1.5,6);

 \draw [fill=blue] (0,0) rectangle (.5,.5);
 \draw [fill=red] (0,0.5) rectangle (.5,1);
 \draw [fill=green] (0,1) rectangle (.5,1.5);

 \draw [fill=black!10!white] (0.5,0)--(1,0)--(1,6)--(0.5,1.5) -- cycle;
 \draw (0,0.5)--(0.5,0.5);
\draw (0,1)--(0.5,1);
\draw (1,2)--(1.5,2);
\draw (1,4)--(1.5,4);
\end{scope}
\end{tikzpicture}
\end{centering}

\vspace{1.5 cm}

\begin{centering}
\begin{tikzpicture}[yscale=.8]
 \draw [fill=red] (0,0) rectangle (0.5,1);
 \draw [fill=yellow] (0,2) rectangle (0.5,1);
 \draw [fill=black!10!white] (0.5,0)--(1,0)--(1,6)--(0.5,2) -- cycle;
 \draw [fill=blue] (1,0) rectangle (1.5,3);
 \draw [fill=green] (1,3) rectangle (1.5,6);
\draw (0,0.5)--(0.5,0.5);
\draw (0,1.5)--(0.5,1.5);
\draw (1,1.5)--(1.5,1.5);
\draw (1,4.5)--(1.5,4.5);

\begin{scope}[xshift=\myspace]
 \draw [fill=red] (0,0) rectangle (0.5,1);
 \draw [fill=red] (0,2) rectangle (0.5,1);
 \draw [fill=black!10!white] (0.5,0)--(1,0)--(1,6)--(0.5,2) -- cycle;
 \draw [fill=blue] (1,0) rectangle (1.5,3);
 \draw [fill=green] (1,3) rectangle (1.5,6);
 \draw (0,0.5)--(0.5,0.5);
\draw (0,1.5)--(0.5,1.5);
\draw (1,1.5)--(1.5,1.5);
\draw (1,4.5)--(1.5,4.5);
\end{scope}
\begin{scope}[xshift=2*\myspace]
 \draw [fill=red] (0,0) rectangle (0.5,1);
 \draw [fill=red] (0,2) rectangle (0.5,1);
 \draw [fill=black!10!white] (0.5,0)--(1,0)--(1,6)--(0.5,2) -- cycle;
 \draw [fill=blue] (1,0) rectangle (1.5,1.5);
 \draw [fill=yellow] (1,1.5) rectangle (1.5,6);
 \draw (0,0.5)--(0.5,0.5);
\draw (0,1.5)--(0.5,1.5);
\draw (1,1.5)--(1.5,1.5);
\draw (1,4.5)--(1.5,4.5);
\draw (1,3)--(1.5,3);
\end{scope}
\begin{scope}[xshift=3*\myspace]
 \draw [fill=green] (0,0) rectangle (0.5,1);
 \draw [fill=green] (0,2) rectangle (0.5,1);
 \draw [fill=black!10!white] (0.5,0)--(1,0)--(1,6)--(0.5,2) -- cycle;
 \draw [fill=blue] (1,0) rectangle (1.5,1.5);
 \draw [fill=yellow] (1,1.5) rectangle (1.5,6);
 \draw (0,0.5)--(0.5,0.5);
\draw (0,1.5)--(0.5,1.5);
\draw (1,1.5)--(1.5,1.5);
\draw (1,3)--(1.5,3);
\draw (1,4.5)--(1.5,4.5);
\end{scope}
\begin{scope}[xshift=4*\myspace]
 \draw [fill=green] (0,0) rectangle (0.5,1);
 \draw [fill=green] (0,2) rectangle (0.5,1);
 \draw [fill=black!10!white] (0.5,0)--(1,0)--(1,6)--(0.5,2) -- cycle;
 \draw [fill=yellow] (1,1.5) rectangle (1.5,6);
 \draw [fill=red] (1,0) rectangle (1.5,1.5);
 \draw [fill=red] (1,3) rectangle (1.5,4.5);
 \draw (0,0.5)--(0.5,0.5);
\draw (0,1.5)--(0.5,1.5);
\draw (1,1.5)--(1.5,1.5);
\draw (1,4.5)--(1.5,4.5);
\end{scope}
\begin{scope}[xshift=5*\myspace]
 \draw [fill=green] (0,0) rectangle (0.5,1);
 \draw [fill=green] (0,2) rectangle (0.5,1);
 \draw [fill=blue] (0,0) rectangle (0.5,0.5);
 \draw [fill=blue] (0,1) rectangle (0.5,1.5);
 \draw [fill=black!10!white] (0.5,0)--(1,0)--(1,6)--(0.5,2) -- cycle;
 \draw [fill=yellow] (1,1.5) rectangle (1.5,6);
 \draw [fill=red] (1,0) rectangle (1.5,1.5);
 \draw [fill=red] (1,3) rectangle (1.5,4.5);
 \draw (0,0.5)--(0.5,0.5);
\draw (0,1.5)--(0.5,1.5);
\draw (1,1.5)--(1.5,1.5);
\draw (1,4.5)--(1.5,4.5);
\end{scope}
\end{tikzpicture}
\end{centering}

\vspace{1.5 cm}

\begin{centering}
\begin{tikzpicture}[yscale=.8]
 \draw [fill=red] (0,0) rectangle (0.5,2);
 \draw [fill=yellow] (0,4) rectangle (0.5,2);
 \draw [fill=black!10!white] (0.5,0)--(1,0)--(1,6)--(0.5,4) -- cycle;
 \draw [fill=blue] (1,0) rectangle (1.5,3);
 \draw [fill=green] (1,3) rectangle (1.5,6);
 \draw (0,3)--(0.5,3);
\draw (0,1)--(0.5,1);
\draw (1,1.5)--(1.5,1.5);
\draw (1,4.5)--(1.5,4.5);
\begin{scope}[xshift=\myspace]
 \draw [fill=red] (0,0) rectangle (0.5,2);
 \draw [fill=red] (0,4) rectangle (0.5,2);
 \draw [fill=black!10!white] (0.5,0)--(1,0)--(1,6)--(0.5,4) -- cycle;
 \draw [fill=blue] (1,0) rectangle (1.5,3);
 \draw [fill=green] (1,3) rectangle (1.5,6);
 \draw (0,3)--(0.5,3);
\draw (0,1)--(0.5,1);
\draw (1,1.5)--(1.5,1.5);
\draw (1,4.5)--(1.5,4.5);
\end{scope}
\begin{scope}[xshift=2*\myspace]
 \draw [fill=red] (0,0) rectangle (0.5,2);
 \draw [fill=red] (0,4) rectangle (0.5,2);
 \draw [fill=black!10!white] (0.5,0)--(1,0)--(1,6)--(0.5,4) -- cycle;
 \draw [fill=yellow] (1,0) rectangle (1.5,3);
 \draw [fill=yellow] (1,3) rectangle (1.5,6);
 \draw (0,3)--(0.5,3);
\draw (0,1)--(0.5,1);
\draw (1,1.5)--(1.5,1.5);
\draw (1,4.5)--(1.5,4.5);
\end{scope}

\begin{scope}[xshift=3*\myspace]
 \draw [fill=green] (0,1) rectangle (0.5,4);
 \draw [fill=blue] (0,3) rectangle (0.5,2);
 \draw [fill=blue] (0,1) rectangle (0.5,0);
 \draw [fill=black!10!white] (0.5,0)--(1,0)--(1,6)--(0.5,4) -- cycle;
 \draw [fill=yellow] (1,0) rectangle (1.5,3);
 \draw [fill=yellow] (1,3) rectangle (1.5,6);
 \draw (0,3)--(0.5,3);
\draw (0,1)--(0.5,1);
\draw (1,1.5)--(1.5,1.5);
\draw (1,4.5)--(1.5,4.5);
\end{scope}

\begin{scope}[xshift=4*\myspace]
 \draw [fill=green] (0,1) rectangle (0.5,4);
 \draw [fill=blue] (0,3) rectangle (0.5,2);
 \draw [fill=blue] (0,1) rectangle (0.5,0);
 \draw [fill=black!10!white] (0.5,0)--(1,0)--(1,6)--(0.5,4) -- cycle;
 \draw [fill=yellow] (1,1.5) rectangle (1.5,6);
 \draw [fill=red] (1,0) rectangle (1.5,1.5);
 \draw [fill=red] (1,3) rectangle (1.5,4.5);
 \draw (0,3)--(0.5,3);
\draw (0,1)--(0.5,1);
\draw (1,1.5)--(1.5,1.5);
\draw (1,4.5)--(1.5,4.5);
\end{scope}
\end{tikzpicture}
\end{centering}
 \captionsetup{width=.8675\textwidth}
 \caption{Examples, for $k=4$, of recolouring sequences between extremal colourings of $K_{p,q}$ in the 3 regimes.
 Each coloured rectangle shows the proportion of vertices coloured in a particular colour.
 Each grey area indicates the edges of $K_{p,q}$. Top: An optimal recolouring sequence from $\a$ to $\b$ when $a=1$. The number of steps needed is shown in \cref{fig:graph_of_functions} as $\mathcal{B}$. Middle: One candidate for an optimal recolouring sequence from $\a$ to $\b$ when $a=\floorfrac{k}{2}$. The number of steps needed is shown in \cref{fig:graph_of_functions} as $\mathcal{G}$. Bottom: The other candidate for an optimal recolouring sequence from $\a$ to $\b$ when $a=\floorfrac{k}{2}$. The number of steps needed is shown in \cref{fig:graph_of_functions} as $\mathcal{O}$.}
 \label{fig:SwitchColours}
\end{figure}
To prove the upper bounds in Theorem~\ref{thr:main_diamCk(Kpq)}, we consider a few candidate recolouring sequences (constructed inductively). We use an averaging argument to show, for each pair of colourings $\a,\b$, not just those defined above, that one of these candidates recolours $\a$ to $\b$ in a number of steps at most the bound in Theorem~\ref{thr:main_diamCk(Kpq)}; examples of these candidate recolouring sequences, when $k=4$, (applied on a pair of extremal colourings $\a, \b$) are shown in \cref{fig:SwitchColours}.

To prove the lower bound in Theorem~\ref{thr:main_diamCk(Kpq)}, we show further that indeed no recolouring sequence can do better.
We remark briefly about the error terms $g_1(k,p,q)$ and $g_2(k,p,q)$. When $p$ and $q$ are both multiples of $\floorfrac{k}2\ceilfrac{k}2=\floorfrac{k^2}4$, these error terms vanish. In the general case, they arise because we cannot quite use all colours ``equally'' on a part when the number of colours does not divide the order of that part. (We suspect that $g_1(k,p,q)=O(k)$, but our efforts to prove this encountered various technical complications.)
Moreover, the error term $g_2(k,p,q)$ becomes $1$ or $0$ when $q$ is sufficiently small compared to $k$. Similarly, as proved in~\cite[Thm.~20]{CCC24}) the diameter of $\C_k(K_{p,q})$ equals $2p+q$ whenever $p \leq q \leq k-2$.

As Theorem~\ref{thr:main_diamCk(Kpq)} illustrates, determining the diameter of $\C_k(G)$ for graphs $G$ in even a simple class like complete bipartite graphs is already quite involved. This example hints at the complexity of the problem in general. 
In this light, proving that $\diam C_k(G)=n(G)+\mu(G)$ for all $k\ge \Delta(G)+2$, and all graphs $G$, (see \cite[Conj.~1]{CCC24}) would be very interesting. When we restrict to regular graphs, this problem becomes Regular Cereceda's Conjecture~\cite[Conj.~3]{CCC24}.
In some sense (expressed in the maximum degree), once $k$ is sufficiently large to guarantee connectedness of $\C_k(G)$, the diameter of $\C_k(G)$ would be known as well. For smaller $k,$ one can expect many different regimes, and it seems impossible to state a precise answer in general.

\section{Optimal Renaming Lemma}
\label{opt-sec}

Bonamy and Bousquet~\cite[Lemma~5]{BB18} proved a variant of the Optimal Renaming Lemma 
below; for clarity, we call their version the Rough Renaming Lemma.  It has been reused in several papers as a subroutine. 
In a typical application, we are given two colourings $\alpha$ and $\beta$ that can be efficiently recoloured to some carefully chosen colourings $\alpha', \beta'$ that induce the same partition into colour classes. The Rough Renaming Lemma is then used as a `finishing blow' to recolour $\alpha'$ to $\beta'$, thus establishing a short recolouring sequence between $\alpha$ and $\beta$.
But the upper bound in the Rough Renaming Lemma, of $2n$, is suboptimal. 
In fact, this recolouring can be done with only $\lfloor 3n/2\rfloor$ steps, which is sharp. 
So reusing the Rough Renaming Lemma is wasteful, especially in papers that prove a linear bound on the 
diameter of the reconfiguration graph. 
In a sense, the ideas needed to prove the Optimal Renaming Lemma were already present in the argument of Bonamy and Bousquet, 
but authors of subsequent papers seem to be unaware. By making this improvement explicit, we indirectly strengthen the upper bounds in 
the papers that have used this Rough Renaming Lemma.

\begin{ORL}
Let $G$ be a graph on $n$ vertices, fix $k\geq\chi(G)$, and let $\alpha$ and $\beta$ be two proper $k$-colourings of $G$ that induce the same partition of vertices into colour classes $C_1,\ldots, C_k$. If $\ell\geq k+1$, then $\alpha$ can be recoloured to $\beta$ in $\mathcal{C}_{\ell}(G)$ in at most $\lfloor 3n/2 \rfloor$ steps, with each vertex recoloured at most $2$ times.
\end{ORL}

\begin{proof}
We define an auxiliary loopless digraph $D$ with vertex set $V(D):= \{C_1,\ldots, C_k \}$, and with $(C_i,C_j)$ a directed edge of $D$ if and only if $\alpha(C_i) = \beta(C_j)$ and $i \neq j$. 
By construction, the vertices of $D$ have pairwise distinct colours, so each vertex has 
in-degree and out-degree both at most $1$. So the vertices of $D$ can be partitioned into directed paths and directed cycles. Now we sequentially recolour the vertices of each directed path and cycle.
For a directed path $C_{i_1},\ldots,C_{i_t}$ (possibly with $t=1$), we simply recolour $C_{i_1}$ to $\beta(C_{i_1})$, then recolour $C_{i_2}$ to $\beta(C_{i_2})$, etcetera, finally recolouring $C_{i_t}$ to $\beta(C_{i_t})$. This takes $ \sum_{j=1}^{t} |C_{i_j}|$ recolouring steps.
Note that every recolouring step induces an updated directed graph $D$. But this is of no consequence for us since every colour class $C_i$ that is currently coloured $\beta(C_i)$ has in-degree and out-degree in $D$ both $0$.

For a directed cycle $C_{i_1},\ldots, C_{i_t}$ with $t\geq 2$, we assume without loss of generality that $C_{i_1}$ has minimum size among $C_{i_1},\ldots,C_{i_t}$. We first recolour $C_{i_1}$ to some free colour in $[\ell] \setminus \gamma(V(G))$, where $\gamma$ denotes the current colouring. Note that this cannot create a new arc incident to the current directed cycle.
Next, we proceed along the directed cycle, starting by recolouring $C_{i_2}$ to $\beta(C_{i_2})$ and ending by recolouring $C_{i_t}$ to $\beta(C_{i_t})$. To finish, we recolour $C_{i_1}$ again, this time to $\beta(C_{i_1})$. In total, the number of recolouring steps we use is $|C_{i_1}| + \sum_{j=1}^{t} |C_{i_j}|$. This is at most $\lfloor 3/2 \cdot \sum_{j=1}^{t} |C_{i_j}| \rfloor$ because $t\geq 2$ and we chose $C_{i_1}$ to be of minimum size. 
In this way we can recolour every directed path and cycle, one after the other. 
So the total number of recolouring steps we use is at most $\lfloor 3/2 \cdot \sum_{i=1}^{k} |C_{i}| \rfloor = \lfloor 3n/2 \rfloor$. Moreover, every vertex is recoloured only once or, if it belonged to the smallest colour class of its directed cycle, twice.
\end{proof}

\begin{cor}
The upper bound in~\cite[Thm.~$6$]{BCM24}, for $(P_3+P_1)$-free graphs, improves from $6n$ to $\lfloor(4+1/2)n\rfloor$; and the upper bound in~\cite[Thm.~$7$]{BCM24}, for $(2K_2,C_4)$-free graphs, improves from $4n$ to $\lfloor 7n/2 \rfloor$.
\end{cor}
\begin{proof}
In these proofs, the authors use the Rough Renaming Lemma,
which promises a recolouring sequence of length at most $2n$, rather than of length at most $\lfloor 3n/2 \rfloor$. So we gain $\lceil n/2\rceil$ steps each time the lemma is applied. To prove Theorem $6$ it is applied 3 times, while to prove Theorem $7$ it is applied 1 time.
\end{proof}

\begin{cor}
In~\cite{BCH24} 
its authors prove that for every $2K_2$-free bipartite $n$-vertex graph $G$, and for every $\ell \geq \chi(G)+1$, we have $\diam \C_{\ell}(G)\le 4n$. That proof uses the Rough Renaming Lemma once. 
By instead using The Optimial Renaming Lemma, we improve the bound to $\lfloor 7n /2\rfloor$.
In the same paper, the authors also prove that for every $2K_2$-free $n$-vertex graph $G$, and for every $\ell \geq \chi(G)+1$, 
we have $\diam\C_{\ell}(G)\le 14n$. Again, they use the Rough Renaming Lemma once, so the bound now improves 
to $\lfloor(13+ 1/2)n \rfloor$.
\end{cor}

\begin{cor}
In~\cite{BC25}, 
the Rough Renaming Lemma is used once to show that for every \{triangle, co-diamond\}-free graph $G$ and for every $\ell \geq \chi(G)+1$, we have $\diam \C_{\ell}(G)\le 6n$, unless $G$ is isomorphic to $K_{\ell,\ell}$ minus a perfect matching. 
The Optimal Renaming Lemma improves this bound to $\lfloor(5+1/2)n\rfloor$.
\end{cor}

\begin{cor}
In~\cite{merkel} 
it is shown that for every $3K_1$-free graph $G$, we have $\diam\C_{\chi(G)+1}(G)\le 4n$. The Rough Renaming Lemma is used once, so the Optimal Renaming Lemma improves this bound to $\lfloor(3+1/2)n\rfloor$.
\end{cor}

\section{Distances in reconfiguration graphs of complete bipartite graphs}\label{sec:diam_Ck(Kpq)}
\label{Kpq-sec}
In this section, we consider $k$-colourings of $K_{p,q}$ and prove~\cref{thr:main_diamCk(Kpq)}.
To begin, we prove the lower bound. 
\subsection{Lower Bounds for \texorpdfstring{$K_{p,q}$}{K-p,q}
}
\label{lower:ssec}

It is trivial to get a lower bound of the order of the graph, $p+q$, by constructing two colourings that differ on each vertex.
To improve this bound, our basic idea is to construct two $k$-colourings that require the colours on many pairs of vertices (in opposite parts) to ``swap''; that is, the old colour of vertex 1 is the new colour of vertex 2 and vice versa. For every such pair, a recolouring step must be ``wasted'', since the first vertex in the pair to be recoloured cannot immediately receive its desired new colour.

\begin{proof}[Proof of the lower bound in~\cref{thr:main_diamCk(Kpq)}]
Fix integers $a, b, k$ with $1 \le a\le b$ and $k=a+b$. For convenience, we assume through most of the proof that $ab$ divides both $p$ and $q$; however, near the end we comment about how to remove this assumption.
In fact we will mainly care about the cases that $(a,b)\in\{(1,k-1),(\floorfrac{k}2,\ceilfrac{k}2)\}$, but for the time being we consider the more general case above.
We define the colourings $\a$ and $\b$ using $ab$ pairs of sets $U_{i,j}, V_{i,j}$ of sizes (respectively) $\frac{p}{ab}$ and $\frac{q}{ab}$, 
with $i \in [a]$ and $j \in [a+1,k]$,
where $U$ and $V$ are the unions of these sets. 
 The colourings $\a$ and $\b$ are defined by $\a(U_{i,j})=i=\b(V_{i,j})$ and $\b( U_{i,j})=j=\a(V_{i,j})$.
 For convenience, we also let $U_i := \cup_{j \in [a+1,k]} U_{i,j}$ and $V_j := \cup_{i \in [a]} V_{i,j}$.

Since the colours on $U$ and $V$ swap from $\a$ to $\b$, there is always a first moment that $U$ resp. $V$, and in particular $U_i$ 
(or $V_j$), is completely coloured with colours different from $i$ (resp. $j$). We call this the time when colour $i$ or colour $j$ is \emph{erased}. (We do not care about a colour disappearing and reappearing on a part, since we are only proving a lower bound.) 
We order the sets $U_i$ (renaming colours if needed) so that colour 1 is erased from $U_1$ before colour 2 is erased from $U_2$, which happens before colour 3 is erased from $U_3$, etc. Similarly, if $j<j'$, then we assume that colour $j$ is erased from $V_j$ before colour $j'$ is erased from $V_{j'}$.

When colour $i$ gets erased from $U_i$, let $f(i)$ denote the number of the $b$ colours on $V$ that are not yet erased. 
So there are at least $f(i)$ sets $U_{i,j}$ which have been recoloured entirely with colours different from their corresponding target
colour $j$. Analogously, we define $f(j)$, for each $j \in [a+1,k]$, as the number of colours that are not yet erased on $U$ 
when $j$ is erased from $V$ (for the first time).

For each pair $(i,j) \in [a] \times [a+1,k]$, one of the two colours $i$ and $j$ must be erased first, when the other is not yet erased.
Since we do not consider reoccurence of a colour after being erased,
this implies $\sum_{h=1}^{k} f(h)=ab$. More specifically, for all distinct $i\in [a]$ and $j\in [a+1,k]$, we define $X_{i<j}$ to be 1 if colour $i$ is erased from $U$ before colour $j$ is erased from $V$ and define it to be 0 otherwise; and we define $X_{j<i}$ as $1-X_{i<j}$. 
Now we get
$\sum_{h=1}^kf(h)
=\sum_{h=1}^af(h)+\sum_{h=a+1}^bf(h) 
= \sum_{h=1}^a\sum_{\ell=a+1}^kX_{h<\ell}+\sum_{\ell=a+1}^k\sum_{h=1}^aX_{\ell<h} 
= \sum_{h=1}^a\sum_{\ell=a+1}^k(X_{h<\ell}+X_{\ell<h}) 
= \sum_{h=1}^a\sum_{\ell=a+1}^k1 = a(k-(a+1)+1) = ab$.

When colour $i$ is erased from $U$, for each of the $f(i)$ colours $j$ that are not yet erased from $V$, the vertex subset $U_{i,j}$
has been completely recoloured with colours other than $j$, so all its vertices will need to be recoloured again to reach $\b$.
An analogous argument works when each colour $j$ is erased from $V$.
 Thus, a lower bound for the number of recolourings needed to transform $\a$ into $\b$ is 
 \begin{equation}\label{eq:LB_recol}
 p+q + \sum_{i=1}^a f(i) \frac{p}{ab} + \sum_{j=a+1}^k f(j) \frac{q}{ab} = 2p+q + \sum_{j=a+1}^k f(j) \frac{q-p}{ab}.
 \end{equation}

Before continuing, we mention that if $q\ge kp$, then we let $(a,b):=(1,k-1)$; and if $q < kp$, then we let $(a,b):=(\lfloor k/2\rfloor, \lceil k/2\rceil)$.
Depending on the order in which the colours were erased, we get different lower bounds.
Recall from above that the first colour erased is either $1$ or $a+1$, and the last colour erased is either $a$ or $k$.
Thus, we are in one of the following two cases.

\textbf{Case 1: $\bm{a+1}$ is the first colour erased or $\bm{a}$ is the last colour erased.}
 In the first case $f(a+1)=a$, and in the second case $f(j)\ge 1$ for every $a+1 \le j \le k$.
 In both cases $\sum_{j=a+1}^k f(j) \ge a$.
 Hence the number of recolourings needed is at least $2p+q+ a \frac{q-p}{ab}= 2p+q + \frac{q-p}{b}$.
If $q\ge kp$, then $b=k-1$, so we have the desired bound. And if $q\le \ceilfrac{k}{2}p$, then $b=\ceilfrac{k}2$ so we again have the desired bound. Finally, suppose that $\ceilfrac{k}2p<q<kp$. In this case, substituting $q > \ceilfrac{k}2p$ shows that the above bound $2p+q+\frac{q-p}{\ceilfrac{k}2}$ exceeds the desired bound: 
$\frac{q-p}{\ceilfrac{k}2}- \left(p+\frac{q-kp}{\floorfrac{k^2}4}\right) \ge \frac{\ceilfrac{k}2p-p}{\ceilfrac{k}2} - \left(p+\frac{\ceilfrac{k}2p-kp}{\floorfrac{k^2}4}\right) = \left(p-\frac{p}{\ceilfrac{k}2}\right)-\left(p-\frac{p\floorfrac{k}2}{\ceilfrac{k}2\floorfrac{k}2}\right)=0$.
In this case, we are done. So we now assume that we are in Case 2, below.
 
\textbf{Case 2: $\bm{1}$ is the first colour erased and $\bm{k}$ is the last one.}
 We define $x$ and $y$ as follows. Let $a+1$ be the $(x+1)^{th}$ colour to be erased;
 that is, we first erase $x$ colours on $U$.
 Let $a$ be the $(k-y)^{th}$ colour to be erased; that is, 
 $y$ is the number of colours on $V$ that we erase after erasing colour $a$ from $U$.

 Note that when we erase from $U$ colours $1$ up to $x$, we must recolour all vertices that used these colours with colours 
 from $[a]$.
 All of these vertices must be recoloured again, before the colours $k-y+1$ up to $k$ are erased from $V$ (since all colours in $[a]$ are erased before any colours in $[k-y+1,k]$). 
 This implies that in addition to~\eqref{eq:LB_recol}, we have at least $xy \frac{p}{ab}$ more recolourings.
 Also, $f(a+1)=a-x$ and $f(j)\ge 1$ for all $j\in[a+2,k-y]$. Thus $\sum_{j=a+1}^k f(j) \ge (a-x)+(k-y-(a+1))=k-x-y-1$.
 As a lower bound on the number of recolourings, we thus get 
\begin{align}
2p+q+ (k-x-y-1) \frac{q-p}{ab}+xy \frac{p}{ab}.\label{gen-lb}
\end{align}

Because we are in Case 2, we know that $a\ge 2$, since we first erase colour 1, and all vertices that were coloured 1 
must be recoloured with other colours in $[a]$. This implies that $q < kp$ and thus that 
$a= \floorfrac k2$ and $b=\ceilfrac k2$. By definition, $x\in[a]$ and $y\in[b]$. But in fact we may assume $x\in[a-1]$.
(If instead we have $x=a$, then erasing colour $a$ from $U$ will require reintroducing to $U$ some colour in $[a-1]$, before erasing
any colour from $V$. Such a sequence can always be ``shortcut'' to a sequence that is no longer but that has $x\in[a-1]$.) Similarly,
we assume $y\in[b-1]$. So we must show that the minimum value of~\eqref{gen-lb}, when $x\in[a-1]$ and $y\in[b-1]$ is at least
the value claimed in the theorem. We have 2 subcases, depending on whether or not 
$q\ge \ceilfrac{k}2p$.

First suppose that $q\ge \ceilfrac{k}2p$.
 Since $\max\{x,y\} \le b-1 \le \ceilfrac {k}{2}-1\le \frac{q-p}{p}$, we see that $(k-x-y-1) \frac{q-p}{ab}+xy \frac{p}{ab}$ is non-increasing in $x$ and $y$; specifically, when we increase $x$ or $y$, the increase in the second term is at most the decrease in 
the first term. So the minimum is attained when both $x$ and $y$ are maximal, i.e. $x=a-1$ and $y=b-1$.
 In that case $(k-x-y-1)\frac{q-p}{ab}+xy \frac{p}{ab}= \frac{q-p}{ab}+(ab-k+1) \frac{p}{ab}
=p+\frac{q-kp}{ab} =p+\frac{q-kp}{\floorfrac{k^2}4}$.

Now suppose instead that $q < \ceilfrac{k}2p$.
Consider an arbitrary pair $(x,y)$. If $x\notin\{1,a-1\}$, then we can either increase or decrease $x$ without increasing the value of \eqref{gen-lb}; by repeating this process, we can assume that $x\in\{1,a-1\}$ and, similarly, that $y\in\{1,b-1\}$. 

Due to the symmetry of~\eqref{gen-lb} in $x$ and $y$, evaluating~\eqref{gen-lb} with $(x,y)=(a-1,1)$ gives the same value as evaluating with $(x,y)=(1,a-1)$. So the case $(x,y)=(a-1,1)$ is settled once we have settled the case $(x,y)\in\{(1,b-1),(1,1)\}$.
Thus, we assume that $(x,y)\in\{(1,1),(1,b-1),(a-1,b-1)\}$.
When $(x,y)=(a-1,b-1)$, as in the previous paragraph, we get the lower bound $p+\frac{q-kp}{ab}$. Since $q<\ceilfrac{k}2p=bp$, we have $aq-q = (a-1)q<(a-1)bp = abp-bp$. Thus, we get $p+\frac{q-kp}{ab} = \frac{pab + q - kp}{ab} > \frac{(aq-q+bp) + q - kp}{ab} = \frac{aq-ap}{ab} = \frac{q-p}b = \frac{q-p}{\ceilfrac{k}2}$, as desired.
When $(x,y)=(1,b-1)$, we get $(k-x-y-1) \frac{q-p}{ab}+xy \frac{p}{ab}=(a-1)\frac{q-p}{ab}+\frac{(b-1)p}{ab}=
\frac{q-p}{b}+\frac{bp-q}{ab}>\frac{q-p}{b}.$
 When $(x,y)=(1,1)$, the conclusion is immediate if $b\ge 3$, since then
 $\frac{(a+b-3)(q-p)}{ab} \ge \frac{q-p}{b}$.
 Instead suppose $a=b=2$. Now \eqref{gen-lb} gives the lower bound $2p+q + \frac{q-p}4+\frac{p}4=2p+q+\frac{q}4$.
By assumption $2p = \ceilfrac{k}2p > q$. Thus, $\frac{q}4 > \frac{q}4+\frac{q-2p}4 = \frac{q-p}2 = \frac{q-p}{\ceilfrac{k}2}$. 
\smallskip

The paragraph above finishes the proof when $ab$ divides both $p$ and $q$. 
Now we handle the general case. \begin{itemize}
 \item We begin with the first regime; that is, when $q\ge kp$. We show that in the lower bound we do not lose anything due to rounding except for what comes from the floor (in the statement of the theorem). In the construction of $\a$ and $\b$, we make each set $U_{1,j}$ of size $\ceilfrac{p}{ab}$ or size $\floorfrac{p}{ab}$, and each set $V_{1,j}$ of size $\ceilfrac{q}{ab}$ or size $\floorfrac{q}{ab}$. Furthermore, we take the larger sets (those of size $\ceilfrac{p}{ab}$ or size $\ceilfrac{q}{ab}$) to be the ones with smaller indices. Now it is straightforward to check that $|V_{1,j}|-|U_{1,j}|-\floorfrac{q-p}{ab}\ge 0$ for all $j\in[2,k]$. (For example, it is enough to consider 2 cases based on whether or not $q\pmod b \ge p\pmod b$.)

Recall that when $q\ge kp$, we let $a:=1$ and $b:=k-1$. Thus, we must be in Case 1; that is, colour 1 is not the first colour erased. In other words, $X_{j<1}=1$ for some $j\in[2,k]$. Now the argument proving~\eqref{eq:LB_recol} can be refined to give the following lower bound:
\begin{align*}
d(\a,\b) & = \sum_{j=2}^k\left(|U_{1,j}|+|V_{1,j}| + (|U_{1,j}|X_{1<j}+|V_{1,j}|X_{j<1})\right) \\
& =|U|+|V|+\sum_{j=2}^k\left((|U_{1,j}|X_{1<j}+|U_{1,j}|X_{j<1})+(|V_{1,j}|-|U_{1,j}|)X_{j<1}\right) \\
& = 2|U|+|V| + \sum_{j=2}^k(|V_{1,j}|-|U_{1,j}|)X_{j<1}\\
& \ge 2p + q + \min_{j\in[2,k]}(|V_{1,j}|-|U_{1,j}|) \\
&\ge 2p+q+\floorfrac{q-p}{b}.
\end{align*}

\item Now we consider the second regime; in particular, $q \ge \ceilfrac k2 p$. If $q \equiv p \pmod{ \floorfrac{k^2}{4}}$, then instead of a term $xy \frac{p}{ab}$ (in the calculation above), we have $xy$ terms that are each equal to $\floorfrac{p}{ab}$ or $\ceilfrac{p}{ab}$.
So the number of terms that are rounded is
$xy\le(a-1)(b-1)= ab-k+1$; thus, $g_1(k,p,q)< ab-k+1.$
Next, we consider the case when $q \not\equiv p \pmod{ \floorfrac{k^2}{4}}$.
Let $q'$ be the largest integer which is at most 
$q$ and satisfies $q' \equiv p \pmod{ \floorfrac{k^2}{4}}$. In that case, we obtain a lower bound on the diameter observing that $\diam \C_k(K_{p,q})\ge \diam \C_k(K_{p,q'})+(q-q')$; we let the $q-q'$ additional vertices satisfy $\a(v) \not= \b(v)$.
By comparing the main terms in the estimate of the diameter, we can conclude an upper bound on $g_1(k,p,q).$ For this, we check two cases.

If $\ceilfrac{k}2 p \le q'$, then
we get $g_1(k,p,q) \le g_1(k,p,q')+ 1$, since $\floorfrac{q-kp}{ab}-\floorfrac{q'-kp}{ab} \le 1$.
So assume instead that $q'< \ceilfrac{k}2 p \le q$. 
We have analogously $g_1(k,p,q) \le g_1(k,p,\ceilfrac{k}2 p) +1=g_2(k,p,\ceilfrac{k}2 p) +1\le \floorfrac k2 +1. $
For this equality, note that in~\cref{thr:main_diamCk(Kpq)} the bounds for the second and third regimes are equal when $q=\ceilfrac{k}2 p$.  For this second inequality, we applied the upper bound for $g_2$, which we prove in the following paragraph.

\item Finally, we consider the third regime. If $q \equiv p \pmod{ \floorfrac{k^2}{4}},$ 
then we can ensure that
$\abs{V_{i,j}}-\abs{U_{i,j}}=\frac{q-p}{ab}$ for every $(i,j) \in [a] \times [a+1,k]$.
Now the analysis is nearly identical to that above when $p$ and $q$ are both divisible by $\floorfrac{k^2}4$; consequently, $g_2(k,p,q)=0$.
In the more general case, let $q:=q'+r$, where $q' \equiv p \pmod{\floorfrac{k^2}{4}}$ and $ 0 \le r < \floorfrac{k^2}{4}$.
Now $\diam \C_k(K_{p,q}) \ge \diam \C_k(K_{p,q'})+r$.
This holds because any pair of colourings $(\a,\b)$ of $K_{p,q'}$ can be extended to colourings of $K_{p,q}$ such that each of the additional $r$ vertices has different colours in $\a$ and in $\b$.
In the third regime, $\diam \C_k(K_{p,q}) \ge \diam \C_k(K_{p,q'})+r$ implies $g_2(k,p,q) \le \floorfrac{q-p}{b}-\floorfrac{q'-p}{b} \le \ceilfrac{q-q'}{\ceilfrac k2} \le \floorfrac k2.$ \qedhere
\end{itemize}
\end{proof}
Now we can sketch the proof of Proposition~\ref{k=3:prop}.  For convenience, we restate it.
\begin{proptwo}
   For every complete bipartite graph $K_{p,q}$, we have $\diam \C_3(K_{p,q})=\floorfrac{3(p+q)}{2}$.
\end{proptwo}
\begin{proof}
   Let $k=3$.  Note that each of the three upper bounds in Theorem~\ref{thr:main_diamCk(Kpq)} (for easier reference, see Theorem~\ref{thr:main_diamCk(Kpq)_upp} below) is equal to $3(p+q)/2$; these bounds are $2p+q+(q-p)/2$ and $3p+q+(q-3p)/2$ and $2p+q+(q-p)/2$.  
   Since the diameter is always an integer, in all 3 regimes Theorem~\ref{thr:main_diamCk(Kpq)_upp} gives the desired upper bound.  So we just need a matching lower bound.  For this, we note that the lower bound in the previous proof, stated for the case $q\ge kp$, actually works for all values of $p$ and $q$.  (Although when $k\ge 4$ and $q<kp$, it is not sharp.) But here it matches our upper bound for all $p$ and $q$.
\end{proof}

\subsection{Upper Bounds for \texorpdfstring{$K_{p,q}$}{K-p,q}}
\label{upper:ssec}

In the remainder of this section, we prove the upper bounds in Theorem~\ref{thr:main_diamCk(Kpq)}.
First, we prove a proposition and a lemma that will be used in the proofs of the three statements of the theorem.
\begin{defi}
\label{Kpq-defi}
Fix a complete bipartite graph $K_{p,q}$, where $p$ and $q$ are positive integers. 
By symmetry, we assume $p\le q$. Fix an integer $k\ge 3$
and $[k]$-colourings $\a$ and $\b$ of $K_{p,q}$. Denote the parts of $K_{p,q}$ by $U$ and $V$, 
with $|U|=p$ and $|V|=q$. Let $C_1:=\a(V)\cap \b(V)$, let $C_2:=\a(V)\setminus C_1$, and let $C_3:=\b(V)\setminus C_1$.
By symmetry, we assume that $\lvert C_3\rvert \ge \lvert C_2\rvert$; if not, then we swap $\a$ and $\b$.
\end{defi}

\begin{lem}\label{lem:V_uses-all}
Fix positive integers $p,q,k$ and $[k]$-colourings $\a$ and $\b$ as in Definition~\ref{Kpq-defi}. 
If there exists $c\in [k]$ used on $V$ by neither $\a$ nor $\b$, then $d(\a, \b) \le 2p+q.$ 
\end{lem}

\begin{proof}
Recolour each vertex of $U$ with
$c$, recolour each $v\in V$ with $\b(v)$, and recolour each $u\in U$ with $\b(u)$. The number of recolouring steps is at most $p+q+p=2p+q$.
\end{proof}

\begin{remark}
\label{nonempty-rem}
Throughout this section, we assume Definition~\ref{Kpq-defi}. When proving the upper bounds for Theorem~\ref{thr:main_diamCk(Kpq)}, we assume that $\a(V)\cup \b(V) = [k]$, since otherwise we are done by Lemma~\ref{lem:V_uses-all}. That is, $C_1$, $C_2$, $C_3$ partition $[k]$.
Note also that $C_2\ne \emptyset$ and $C_3\ne\emptyset$, since $U$ is coloured under both $\b$ and $\a$.
\end{remark}

\begin{prop}\label{prop:caseC_1emptyset}
Fix integers $p,q,k$, the graph $K_{p,q}$, and $[k]$-colourings $\a$ and $\b$, as in \cref{Kpq-defi}.
Assume that $\a(U)=\b(V)=C_3$ and $\a(V)=\b(U)=C_2,$ where $C_2 \cup C_3=[k]$. If $\lvert C_3 \rvert \ge \lvert C_2 \rvert$,
then $$d(\a,\b) \le q+2p+ \min \left\{\frac{q-p}{|C_3|}, \frac{q-kp}{|C_2||C_3|}+p \right\}.$$
\end{prop}

\begin{proof}
Fix a non-empty subset $C'_3\subset C_3$ (not equal to the whole set $C_3$; this is possible, since $|C_3|\ge \lceil k/2\rceil\ge 2$)
and a non-empty subset $C'_2 \subseteq C_2$ (possibly equal to $C_2$).

We now sketch a possible series of recolourings from $\a$ to $\b.$
First recolour $\a^{-1}(C'_3) \subset U$ (with colours from $C_3 \backslash C'_3$).
Then recolour each vertex $v \in \a^{-1}(C'_2) \subset V$, with its colour $\b(v)$ whenever $\b(v) \in C'_3$, and otherwise recolour each such $v$ with an arbitrary colour in $C'_3 \cup (C_2\setminus C_2')$.
Now recolour all of $U$ with $C'_2$, such that $U$ does not use any colour of $C_3$, and such that each $u \in U$ is coloured with $\b(u)$ if possible.
Next, we recolour each remaining vertex $v\in V$ with $\b(v)$; finally, we recolour each vertex $u\in U$ with $\b(u)$.

In total, the number of steps we used is $\lvert \a^{-1}(C_3')\rvert + \lvert\a^{-1}(C_2')\rvert + \lvert\a^{-1}(C_2\setminus C_2')\rvert + \lvert\a^{-1}(C_2')\cap \b^{-1}(C_3\setminus C_3')\rvert + \lvert\b^{-1}(C_2\setminus C_2')\rvert = p+q+\lvert \a^{-1}(C'_3) \rvert + \lvert \a^{-1}(C'_2) \cap \b^{-1}(C_3 \backslash C'_3) \rvert+\lvert \b^{-1}(C_2 \backslash C'_2)\rvert$.

Rather than fix sets $C_2'$ and $C_3'$, we just fix their sizes $|C_2'|$ and $|C_3'|$.
Let 
$f_2:=\frac{| C'_2|}{\abs{ C_2}}$ and 
$f_3:=\frac{|C_3 \backslash C'_3|}{\abs{ C_3}}$.
Averaging over all possible sets $C'_2$ and $C'_3$ of these fixed sizes, 
the number of steps we use is
$p+q+(1-f_3)p+f_2f_3q+(1-f_2)p=2p+q+f_2f_3q+(1-f_2-f_3)p$.
Doing this for $\lvert C'_2\rvert=\lvert C_2\rvert$ and $\lvert C'_3 \rvert =\lvert C_3 \rvert-1$, as well as for $\lvert C'_2 \rvert =1$ and $\lvert C'_3 \rvert =1$ results in 
$$d(\a,\b) \le 2p+q+\min \left\{\frac{q-p}{|C_3|}, \frac{q}{|C_2||C_3|}+ \left(1-\frac{1}{|C_2|}-\frac{1}{|C_3|} \right)p \right\}.$$

Noting that $\frac{1}{|C_2|}+\frac{1}{|C_3|}=\frac{|C_2|+|C_3|}{|C_2||C_3|}$ and $|C_2|+|C_3|=k$ results in the formulation of the proposition.
\end{proof}

We are now ready to prove the upper bounds of~\cref{thr:main_diamCk(Kpq)}. For ease of presentation, we restate the upper bounds. Since the diameter is always an integer, the bounds also hold when taking the floor function.

\begin{theorem}\label{thr:main_diamCk(Kpq)_upp}
Fix positive integers $p,q,k$. If $p \le q$ and $k\ge 3$, then 
$$\diam\C_k(K_{p,q})\le 
\left.
 \begin{cases}
 2p+q+\frac{q-p}{k-1} & \text{if } q \ge kp \\ \\
 3p+q+\frac{q-kp}{\floorfrac{k^2}{4} } & \text{if } \ceilfrac {k}{2} p \le q \le kp,\\
 \\
 2p+q+\frac{q-p}{\ceilfrac {k}{2} }& \text{if } q \le \ceilfrac {k}{2} p.
 \end{cases}
 \right.$$
\end{theorem}

\begin{proof}
 Fix $K_{p,q}$, $[k]$, $\a$, $\b$, $C_1$, $C_2$, $C_3$ as in Definition~\ref{Kpq-defi}.

 If $\abs{C_1}=0$, then we are done by~\cref{prop:caseC_1emptyset}; below, we check this for the 3 regimes:

 \begin{itemize}
 \item In the first regime, it is sufficient to note that $q-kp \ge 0$ and $\abs{C_3} \abs{C_2} \ge k-1$. Together, these imply that 
 $\frac{q-kp}{|C_2||C_3|}+p \le \frac{q-kp}{k-1}+p=\frac{q-p}{k-1}$.
 \item In the second regime, since $q-kp\le 0$ and $|C_2||C_3|\le \floorfrac{k^2}{4}$, we have
 $$\frac{q-kp}{|C_2||C_3|}+p \le p+ \frac{q-kp}{\floorfrac{k^2}{4} }.$$
 \item In the third regime, it is sufficient to note that $\abs{C_3} \ge \lceil k/2\rceil$ and thus 
 $\frac{q-p}{|C_3|}\le \frac{q-p}{\ceilfrac {k}{2}}$.
 \end{itemize}

 \noindent
 Thus, we assume below that $\abs{C_1}\ge 1$.

 We proceed by induction on $k$.
 The base case is $k=3$, and thus $\abs{C_1}=\abs{C_2}=\abs{C_3}=1.$
 When $k=3$, we also denote by $C_i$ the single colour of each set $C_i$.
 We describe $2$ recolouring sequences; by an averaging argument we will show that at least one of them has length at most $p+q+\frac{p+q}2$ (when $k=3$, all $3$ upper bounds above are equal).

 As a first method to recolour $\a$ into $\b$, we: (i) recolour $\a^{-1}(C_2)$ with $C_1$, (ii) recolour $U$ with $C_2$, and (iii) recolour $\b^{-1}(C_3)$ with colour $C_3.$
This implies
\begin{equation}\label{eq:k=3_1}
d(\a,\b)\le p+ \abs{\a^{-1}(C_2)} + \abs{\b^{-1}(C_3)}.
\end{equation}

As a second method, we: (i) recolour $\a^{-1}(C_1)$ with $C_2$, (ii) recolour each $u\in U$ with $C_1$, 
(iii) recolour each $v\in V$ with $C_3$, (iv) recolour each $u\in U$
with $\b(u)=C_2$, and (v) recolour each $v\in \b^{-1}(C_1)$ with $\b(v)=C_1$. The number of steps this takes is 
$\lvert \a^{-1}(C_1)\rvert + p + q + p + \lvert\b^{-1}(C_1)\rvert$. That is, 
\begin{equation}\label{eq:k=3_2}
 d(\a,\b)\le 2p+q + \lvert\a^{-1}(C_1)\rvert
+\lvert \b^{-1}(C_1)\rvert.
\end{equation}

 Since $|\a^{-1}(C_1)| + |\a^{-1}(C_2)|=|V|=|\b^{-1}(C_1)| + |\b^{-1}(C_3)|$, 
 adding~\eqref{eq:k=3_1} and~\eqref{eq:k=3_2} gives $2d(\a,\b) \le 3(p+q)$, as desired.
 This concludes the base case.

For the induction step, we assume $k \ge 4$ and the statement is true for $k-1$ colours.
For a fixed colour $c \in C_1,$ we let $y_0 := \abs{ \a^{-1}(c) \cap \b^{-1}(c) }$, 
 $y_1 := \abs{ \a^{-1}(c) \setminus \b^{-1}(c) }$, $y_2 := \abs{ \b^{-1}(c) \setminus \a^{-1}(c) }$, 
 and $y:=\abs{ \a^{-1}(c) \cup \b^{-1}(c) }=y_0+y_1+y_2.$
Again we will use two recolouring sequences, which imply the following two statements.

 \begin{claim}
 \begin{equation}\label{eq:constr2} d(\a,\b)\le \diam \C_{k-1}(K_{p,q-y})+ y-y_0\end{equation}
 \end{claim}

 \begin{claimproof}
 We (i) recolour $ \b^{-1}(c) \setminus \a^{-1}(c)$ in colour $c$.
     Now (ii) every vertex $w$ in $U \cup (V \setminus (\b^{-1}(c) \cup \a^{-1}(c)))$ can be recoloured into $\b^{-1}(u)$ using at most $\diam \C_{k-1}(K_{p,q-y})$ steps.
 Finally, we (iii) recolour every $ v \in \a^{-1}(c) \setminus \b^{-1}(c) $ with $\b(v)$.
 In total, the number of steps we use is at most $y_2+\diam \C_{k-1}(K_{p,q-y})+y_1$. 
 \end{claimproof}

 \begin{claim}
 \begin{equation}\label{eq:constr1} d(\a,\b)\le 2p+q+y+y_0\end{equation}
 \end{claim}

 \begin{claimproof}
 We (i) recolour $\a^{-1}(c)$ from $C_2$, (ii) recolour $U$ with $c$, (iii) recolour every 
 $v \in V\setminus \b^{-1}(c)$ with $\b(v)$, (iv) recolour every $v\in \b^{-1}(c)$ with an arbitrary colour from $C_3$, 
 (v) recolour every $u \in U$ with $\b(u)$, and (vi) recolour $\b^{-1}(c)$ with $c$.
 In total, the number of steps we use is $(y_0+y_1)+p+q+p+(y_0+y_2)=2p+q+y+y_0$.
 \end{claimproof}

 With regards to the 3 cases in the theorem,
 we have 3 possibilities for $q$: (a) $q\ge kp$, (b) $kp > q > \lceil k/2\rceil p$, and (c) $\lceil k/2\rceil p \ge q$.
 We also have 3 possibilities for $q-y$: (i) $q-y\ge (k-1)p$, (ii) $(k-1)p > q-y > \lceil (k-1)/2\rceil p$, and 
 (iii) $\lceil (k-1)/2\rceil p \ge q-y$. These possibilities for $q$ and $q-y$ combine below to give us 9 cases: (a.i) through (c.iii), although we handle (a.i) and (b.i) together, and also handle (c.i) and (c.ii) together.

 \begin{itemize}
 \item
 \textbf{(a.i) and (b.i): $\bm{q > \ceilfrac {k}{2}p }$ and $\bm{q-y \ge (k-1)p}$.}
 Adding $(k-2)$ times~\eqref{eq:constr2} to~\eqref{eq:constr1}, then dividing by $k-1$ gives
 $$d(\a,\b) - (2p+q) \le \frac1{k-1}\left[(k-2)\left(-y_0+\frac{q-y-p}{k-2}\right)+y+y_0\right] \le \frac{q-p}{k-1}.$$
 If $q\ge kp$, then this gives the desired bound on $d(\a,\b)$. But if $kp > q > \lceil k/2\rceil p$, then
\begin{equation*}
 \frac{q-p}{k-1} = p + \frac{q-kp}{k-1}\le p+\frac{q-kp}{ \floorfrac{k^2}{4}},
\end{equation*}
 since $k-1=\lfloor (4k-4)/4\rfloor \le \lfloor k^2/4\rfloor$.

 \item \textbf{(a.ii): $\bm{q\ge kp}$ and $\bm{(k-1)p> q-y > \ceilfrac {k-1}2 p}$.}
From~\eqref{eq:constr2}, we have
 $$d(\a,\b) - (2p+q) \le p-y_0+\frac{q-y-(k-1)p}{\floorfrac{(k-1)^2}4}<p
 = \frac{kp-p}{k-1} \le \frac{q-p}{k-1}.$$ 

 \item \textbf{(a.iii): $\bm{q\ge kp}$ and $\bm{\ceilfrac {k-1}2 p \ge q-y}$.}
 From~\eqref{eq:constr2}, we have
 \begin{align*} 
 d(\a,\b)-(2p+q)&\le -y+\frac{q-y-p}{\lceil (k-1)/2\rceil}+y-y_0 \le \frac{\lceil(k-1)/2\rceil p-p}{\lceil(k-1)/2\rceil}\\
 &= p - \frac{p}{\lceil(k-1)/2\rceil} \le p - \frac{p}{k-1} = \frac{pk-p-p}{k-1} \le \frac{q-p}{k-1}.
 \end{align*}

 \item \textbf{(b.ii) $\bm{kp > q > \ceilfrac {k}{2} p}$ and $\bm{(k-1)p > q-y > \ceilfrac {k-1}2 p }$.}

Adding $\floorfrac{(k-1)^2}{4}$ times~\eqref{eq:constr2} to~\eqref{eq:constr1}, then dividing by 
 $\floorfrac{(k-1)^2}{4}+1\le \floorfrac{k^2}{4}$, gives
 \begin{align*}
 d(\a,\b)-(2p+q) &\le \left(\floorfrac{(k-1)^2}{4}\left(p-y+\frac{q-y-(k-1)p}{\lfloor (k-1)^2/4\rfloor}+y-y_0\right)+y+y_0\right)/\left(\floorfrac{(k-1)^2}{4}+1\right)\\
 &\le \left(p\left(\floorfrac{(k-1)^2}{4}+1\right)+q-kp\right)/\left(\floorfrac{(k-1)^2}{4}+1\right) \le p + \frac{q-kp}{\floorfrac{k^2}{4}}.
 \end{align*}

 \item \textbf{(b.iii): $\bm{kp > q > \ceilfrac {k}{2} p}$ and $\bm{\ceilfrac{k-1}2 p\ge q-y}$.}
From \eqref{eq:constr2}, we have 
\begin{align*}
 d(\a,\b) - (2p+q) & \le -y_0+\frac{q-y-p}{\ceilfrac {k-1}2 } 
 \le \frac{\ceilfrac{k-1}2 p - p}{\ceilfrac {k-1}2 } \\
 &= p - \frac{p}{\floorfrac {k}2 } 
 = p - \frac{p\ceilfrac{k}2}{\floorfrac {k^2}4 } 
 \le p + \frac{q-kp}{ \floorfrac{k^2}{4} }.
\end{align*}
 Here the last equality holds because $\floorfrac k2 \ceilfrac k2 = \floorfrac{k^2}{4}$,
 and the last inequality holds because $q-kp \ge \ceilfrac k2 p - kp = - \floorfrac k2 p\ge -\ceilfrac k2 p$.

 \item \textbf{(c.i) and (c.ii): $\bm{\ceilfrac{k}{2}p\ge q}$ and $\bm{q-y > \ceilfrac{k-1}{2}p}$.}
 This case can only occur when $\ceilfrac{k}2 > \ceilfrac{k-1}2$; that is, when $k$ is odd. 
 So let $r:=(k-1)/2.$
 Since $\ceilfrac{k}{2}< k-1$, we must show that $d(\a,\b) \le 3p+q+\frac{q-kp}{\ceilfrac{k^2}4}$.
 (That is, we must be in (c.ii); in fact, (c.i) is impossible.)

 If $d(\a,\b)-(2p+q)\le \frac{q-p}{r+1}$, then we have 
 \begin{align*}
 d(\a,\b)-(2p+q)\le \frac{q-p}{r+1} 
 &= p + \frac{q-p-p\ceilfrac{k}{2}}{\ceilfrac{k}2}
 = p + \frac{q-p\ceilfrac{k}{2}}{\ceilfrac{k}2}+\frac{-p}{\ceilfrac{k}2}\\
 &\le p + \frac{q-p\ceilfrac{k}{2}}{\ceilfrac{k^2}4} +\frac{-p\floorfrac{k}2}{\floorfrac{k^2}4}
 = p + \frac{q-pk}{\floorfrac{k^2}4}.
 \end{align*}
So instead we assume $d(\a,\b)-(2p+q)> \frac{q-p}{r+1}$. By~\eqref{eq:constr1} and~\eqref{eq:constr2}, respectively, this implies the following.

 \begin{equation}\label{eq:1}
 y+y_0> \frac{q-p}{r+1}
 \end{equation}
 \begin{equation}\label{eq:2}
 p-y_0+\frac{q-y-2rp}{r^2}> \frac{q-p}{r+1}
 \end{equation}

 (If we are in (c.ii), then \eqref{eq:2} follows directly from the induction hypothesis. But if we are in (c.i), then we note that the left side of \eqref{eq:2} is greater than the upper bound guaranteed by the induction hypothesis; this can be verified by a bit of algebra, which we omit.) 
 Since we are in (c.i) or (c.ii), we have $pr=p\ceilfrac{k-1}2<q-y=q-(y+y_0)+y_0 < q-\frac{q-p}{r+1} +y_0= \frac{rq+p}{r+1}+y_0$; 
 equivalently,
 \begin{equation}\label{eq:3}
 (r^2+r-1)p<rq+(r+1)y_0.
 \end{equation}

 Adding $r^2$ times~\eqref{eq:2} to~\eqref{eq:1} gives
 \begin{align*}
 r^2 p -(r^2-1)y_0 +q-2rp &\ge \frac{r^2+1}{r+1} (q-p)\\ 
 r^3p+r^2p-(r^2-1)(r+1)y_0+qr+q-2r^2p-2rp &\ge (r^2+1)q-(r^2+1)p\\
 (r^3-2r+1)p &\ge (r^2-r)q+(r^2-1)(r+1)y_0\\
 (r^2+r-1)p &\ge rq +(r+1)^2 y_0.
 \end{align*}
 But this final inequality contradicts~\eqref{eq:3}, which finishes the case.

 \item \textbf{(c.iii) $\bm{\ceilfrac{k}{2}p\ge q}$ and $\bm{\ceilfrac{k-1}{2}p\ge q-y}$.}
 If $k$ is even, then $\ceilfrac{k-1}{2}=\ceilfrac{k}{2}$, so \eqref{eq:constr2} gives
 $$d(\a,\b) -(2p+q) \le -y + \frac{q-y-p}{\ceilfrac {k-1}2} + y - y_0
 \le \frac{q-p}{\ceilfrac {k}2}.$$

 If $k$ is odd, then adding $\frac{k-1}{2}$ times~\eqref{eq:constr2} to~\eqref{eq:constr1}, and dividing by
 $\frac{k+1}2$, gives
 \begin{align*}
 d(\a,\b)-(2p+q) &\le \frac{1}{\frac{k+1}2} \left(\left(\frac{k-1}2\right)\frac{q-y-p}{\frac{k-1}2} -\frac{k-1}{2}y_0+(y+y_0)\right)\\
 & \le \frac{q-p}{\ceilfrac{k}2}-\frac{\frac{k-3}2}{\frac{k+1}2}y_0 \le \frac{q-p}{\ceilfrac{k}2}.
 \end{align*}
\end{itemize}
This finishes the last of the 9 cases in the induction step, and thus concludes the proof. \qedhere
\end{proof}

\section{Non-relationships between \texorpdfstring{$\C_k(G)$}{Ck(G)}, \texorpdfstring{$\C_{k+1}(G)$}{Ckplus1(G)}, and \texorpdfstring{$\C_L(G)$}{CL(G)}}\label{sec:particularbehaviour}
\label{surprise-sec}

In~\cite{BFHR22}, the authors sketched some typical behaviour of configuration graphs.
In particular, they mentioned that by increasing the number of colours, the recolouring sequences typically become shorter and shorter.
Note that two factors are at play.
If the number of colours, $k$, increases, then this also increases both the total number of colourings and the number of possible local changes, the degree of the reconfiguration graph.
The first factor often leads to a higher diameter, but the second to a lower diameter.
Here, we give some exceptions to the typical behaviour.
A classical example from~\cite[Prop.~2]{CvdHJ08} is the following.
\begin{ex}
\label{Kmm-M:ex}
 If $G=K_{m,m} \backslash M$ ($K_{m,m}$ minus a perfect matching), then $\C_k(G)$ is connected if and only if $k \ge 3$ and $k \not=m$.
\end{ex}

The key idea in \cref{Kmm-M:ex} is that if $k=m$, then $G$ has a frozen colouring, namely, the one that uses each colour once on each part; see~\cref{fig:L5_FrozenColouring}. (Otherwise, by pigeonhole principle each part has a colour that is used either 0 times or at least 2 times.)
\Cref{Kmm-M:ex} is the inspiration for more counterexamples.
Bard~\cite{Bard2014, Mutze23} asked whether $\C_k(G)$ being Hamiltonian implies that $\C_{k+1}(G)$ is Hamiltonian.
In~\url{https://github.com/StijnCambie/reconfiguration} (document $\texttt{Check\_Ham}$), it is e.g.~verified that if $G=K_{5,5} \backslash M$ then the reconfiguration graph $\C_3(G)$ is Hamiltonian.
Since $\C_5(G)$ is disconnected and thus not Hamiltonian, 
this answers the question of Bard negatively (either for $k=3$ or for $k=4$, but it is not clear which one).  This counterexample was mentioned in~\cite[Sec.~9.6]{Mutze23}, attributed to the first author of the present paper, but no proof was provided; so the code referenced above provides this proof.

Our next example modifies
$K_{4,4} \backslash M$, 
by replacing one vertex with $3$ (each of which
inherits exactly $2$ of its $3$ initial neighbours ). In each of the $4$-colourings shown (see the right of~\cref{fig:C3G<C4G}), the colours on many of the vertices are frozen.
Not surprisingly, the two 4-colourings are far apart.

\begin{ex}\label{ex:C3vsC4}
 The graph $G$ presented in~\cref{fig:C3G<C4G} is bipartite with order $10$, matching number $4$, and degeneracy $\delta^\star(G)=2.$ 
 For $k \ge 3$, $\C_k(G)$ is connected.
 It satisfies $\diam \C_3(G)=15$ and $\diam \C_4(G)=17,$ thus $\diam \C_k(G) < \diam \C_{k+1}(G)$ for $k=3.$
 So $G$ is a graph for which an additional colour can increase the diameter of the recolouring graph.
 Furthermore, $\rad \C_3(G)=15$ and $\rad \C_4(G)=14.$
\end{ex}

We verify all of these numbers in~\url{https://github.com/StijnCambie/reconfiguration} (document \texttt{NonMonotoneC3vsC4}). 
As we know from~\cite[Cor.~14]{CCC24}, for $k$ sufficiently large, the diameter of the reconfiguration graph stabilises at the minimum possible bound (this is the graph's order plus its matching number).
The graph in~\cref{ex:C3vsC4} illustrates that for small $k$, this behaviour can be harder to predict. 
By simply taking many disjoint copies of it, we obtain infinitely many graphs $G$ with $\frac{\diam \C_{k+1}(G)}{\diam \C_{k}(G)} = 17/ 15$.
We do not know if this ratio can become arbitrarily large, while staying finite.

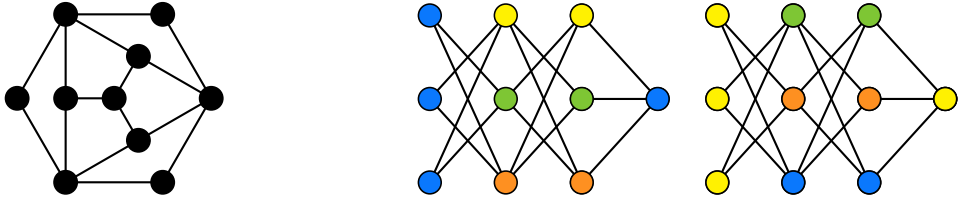
\begin{figure}[ht]
 \centering

\begin{tikzpicture}[ultra thick, scale=0.86]
\def\mydiam{0.18}
\foreach \x in {0,60,...,300}{
\draw[semithick, fill] (\x:1.5) circle (\mydiam);
\draw[thick] (\x:1.5) --(\x+60:1.5) ;
}
\foreach \x in {0,120,240}{
\draw[thick] (0,0)--(\x+60:0.75) --(\x:1.5)--(\x-60:0.75) ;
\draw[semithick, fill] (\x+60:0.75) circle (\mydiam);
}
\draw[semithick, fill] (0,0) circle (\mydiam);

\end{tikzpicture} \quad\quad\quad\quad\quad\quad\quad
 \begin{tikzpicture}[scale=1.01,rotate=90]
\foreach \x/\y in {0/0,-1.1/1,0/1,1.1/1}{
\draw[fill] (\x,\y) circle (0.15);
\draw[thick] (0,0)--(\x,\y);
}

\foreach \x/\y in {-1.1/0,0/1.1,1.1/-1.1,-1.1/1.1,0/-1.1,1.1/0}{
\draw[thick] (\x,1)--(\y,2);
\draw[thick] (\y,2)--(\x,3);
}
\foreach \x in {1,2}{
 \draw[semithick, fill=red] (-1.1,\x) circle (0.15);
 \draw[semithick, fill=green] (0,\x) circle (0.15);
 \draw[semithick, fill=yellow] (1.1,\x) circle (0.15);
}
\foreach \x in {-1.1,0,1.1}{
\draw[semithick, fill=blue] (\x,3) circle (0.15);
}
\draw[semithick, fill=blue] (0,0) circle (0.15);
\end{tikzpicture}
\quad
\begin{tikzpicture}[scale=1.01,rotate=90]
\foreach \x/\y in {0/0,-1.1/1,0/1,1.1/1}{
\draw[semithick, fill] (\x,\y) circle (0.15);
\draw[thick] (0,0)--(\x,\y);
}\draw[semithick, fill] (0,0) circle (0.15);

\foreach \x in {-1.1,0,1.1}{
\draw[semithick, fill] (\x,2) circle (0.15);
\draw[semithick, fill] (\x,3) circle (0.15);
}

\foreach \x/\y in {-1.1/0,0/1.1,1.1/-1.1,-1.1/1.1,0/-1.1,1.1/0}{
\draw[thick] (\x,1)--(\y,2);
\draw[thick] (\y,2)--(\x,3);
}

\foreach \x in {1,2}{
 \draw[semithick, fill=blue] (-1.1,\x) circle (0.15);
 \draw[semithick, fill=red] (0,\x) circle (0.15);
 \draw[semithick, fill=green] (1.1,\x) circle (0.15);
}
\foreach \x in {-1.1,0,1.1}{
\draw[semithick, fill=yellow] (\x,3) circle (0.15);
}
\draw[semithick, fill=yellow] (0,0) circle (0.15);
\end{tikzpicture}
\captionsetup{width=.85\textwidth}
 \caption{Left: A graph $G$ for which $\diam \C_3(G)<\diam \C_4(G)$. Right: Two [4]-colourings of $G$ (drawn differently from the figure on the left) that in $\C_4(G)$ lie at distance $17$.} \label{fig:C3G<C4G}
\end{figure}

\begin{ques}
Fix an integer $k\geq 3$.
Does there exist a sequence of graphs $(G_i)_{i \in \mathbb{N}}$ such that $\diam \C_{k+1}(G_i) < \infty $ for all $i$, but $ \lim_{i \rightarrow \infty}\frac{\diam \C_{k+1}(G_i)}{\diam \C_{k}(G_i)}=\infty$? 
\end{ques}

Interestingly, we were unable to find an analogue of~\cref{ex:C3vsC4} for radius.
An affirmative answer to the following question would imply that the ratio $\frac{\diam \C_{k+1}(G)}{\diam \C_{k}(G)}$ is actually bounded by 2, whenever it is well-defined.

\begin{ques}
Fix an integer $k \ge 3$ and let $G$ be a graph such that $\C_k(G)$ and $\C_{k+1}(G)$ are both connected.
Is it necessarily true that $\rad \C_k(G) \ge \rad \C_{k+1}(G)$?
\end{ques}

Next, we compare the diameter of the $k$-colour reconfiguration graph and the list reconfiguration graph for a $k$-list-assignment. Despite identical lists imposing the most restrictions, we prove that the $k$-colour reconfiguration graph is not extremal; more precisely, for arbitrary $k$-list-assignments $L$, no relationships hold in general between the connectedness (or diameter) of $\C_k(G)$ and that of $\C_{L}(G)$.

\begin{prop}\label{prop:C_k&C_Lunrelated}
 For each of the four statements below, there are examples of graphs $G$, integers $k$, and list-assignments $L$ with $\lvert L(v) \rvert =k$ for every $v \in V(G)$ that satisfy the statement.
 \begin{enumerate}[label=(\alph*)]
 \item $\C_k(G)$ is disconnected, while $\C_L(G)$ is connected
 \item $\diam \C_L(G) < \diam \C_k(G)< \infty$
 \item $\C_L(G)$ is disconnected, while $\C_k(G)$ is connected
 \item $\diam \C_k(G) < \diam \C_L(G)<\infty$
 \end{enumerate}
\end{prop}

\begin{proof}
We prove the first two statements together, and the third and fourth thereafter.

 (a,b) The first two statements are true almost trivially. 
 Take an $n$-vertex graph $G$ and a value $k$ for which $\diam \C_k(G)>n.$
 Let the lists $L(v)$, for all $v \in V$, be disjoint; so $\diam \C_L(G)=n$.
 Here we can choose $\diam \C_k(G)$ to be finite or infinite (equivalently $\C_k(G)$ can be connected or disconnected). For example, with $G=K_{m,m}$, when either $k>2$ or $k=2$.
 By instead letting $G:=P_n$ and $k:=3$, by~\cite[Thm.~3]{BJLPP14} we know that $\frac{\diam \C_k(G)}{\diam \C_L(G)}$ can be arbitrarily large (while being finite).

 (c) Fix $m\ge 4$, let $G$ be the graph $K_{m,m} \backslash M$ for a perfect matching $M$, and let $k:=m-1$.
 We denote the vertices in the two parts by $v_1,\ldots,v_m$ and $w_1,\ldots,w_m$, with $v_iw_i\notin E(G)$ for all $i\in[m]$.
 Let $L(v_i)=L(w_i):=[m]\setminus\{i\}$.
 Then there exist $L$-colourings that are frozen in $\C_L(G)$, e.g. $\varphi(v_i)=\varphi(w_i)=(i+1)\bmod m$; see~\cref{fig:L5_FrozenColouring}. As such, $\C_L(G)$ is disconnected; in contrast, $\C_k(G)$ is connected, as proven in~\cite{CvdHJ08} (the intuition behind this final statement is that on each part some colour must appear at least twice, by pigeonhole principle; so we can recolour the whole part with that colour).

\begin{figure}[ht]
 \centering
 \begin{tikzpicture}[rotate=90][line cap=round,line join=round,>=triangle 45,x=1.5cm,y=1.5cm]
\clip(-3.38,-5.24) rectangle (0.34,-0.7);
\draw [line width=2.5pt] (-3,-1)-- (0,-2)
(-3,-1)-- (0,-3)
(-3,-1)-- (0,-4)
(-3,-1)-- (0,-5)
(-3,-2)-- (0,-1)
(-3,-2)-- (0,-3)
(-3,-2)-- (0,-4)
(-3,-2)-- (0,-5)
(-3,-3)-- (0,-1)
(0,-2)-- (-3,-3)
(-3,-3)-- (0,-4)
(-3,-3)-- (0,-5)
(-3,-4)-- (0,-1)
(0,-2)-- (-3,-4)
(-3,-4)-- (0,-3)
(0,-5)-- (-3,-4)
(-3,-5)-- (0,-4)
(-3,-5)-- (0,-3)
(0,-2)-- (-3,-5)
(0,-1)-- (-3,-5);
\begin{scriptsize}
\draw [fill=yellow, line width=1pt] (-3,-1) circle (5pt);
\draw [fill=yellow, line width=1pt] (0,-1) circle (5pt);
\draw [fill=red, line width=1pt] (-3,-3) circle (5pt);
\draw [fill=red, line width=1pt] (0,-3) circle (5pt);
\draw [fill=blue, line width=1pt] (-3,-5) circle (5pt);
\draw [fill=blue, line width=1pt] (0,-5) circle (5pt);
\draw [fill=green, line width=1pt] (-3,-4) circle (5pt);
\draw [fill=green, line width=1pt] (0,-4) circle (5pt);
\draw [fill=purple, line width=1pt] (-3,-2) circle (5pt);
\draw [fill=purple, line width=1pt] (0,-2) circle (5pt);
\end{scriptsize}
\end{tikzpicture}
 \caption{A frozen colouring of $K_{5,5}\backslash M$}
 \label{fig:L5_FrozenColouring}
 \end{figure}
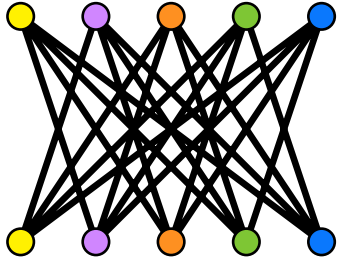
 
 (d) Let $k:=4$, let $G:=K_{18,18}$, and call its parts $U$ and $V$.
 Recall that $54=n(G)+\mu(G)\le \diam \C_4(G)\le 54$, where the upper bound is true by \cref{thr:main_diamCk(Kpq)}.
 So it suffices to construct a 4-assignment $L$ for $G$ such that $\diam \C_L(G) \ge 55$.
 This is what we do.
 
 For every triple $(i,j,\ell)$ with $i \in \{1,3,5\}$, with $j \in \{2,4\}$, and with $\ell \in [5]\setminus\{i,j\}$,
 for unique vertices $u \in U$ and $v \in V$
 we set $L(u)=L(v):=[5]\backslash\{\ell\}$, and 
 define $L$-colourings $\a$ and $\b$ by $\a(u)=i=\b(v)$ and $\b(u)=j=\a(v)$. These colourings are depicted in~\cref{fig:SwitchColoursK18_18}.

 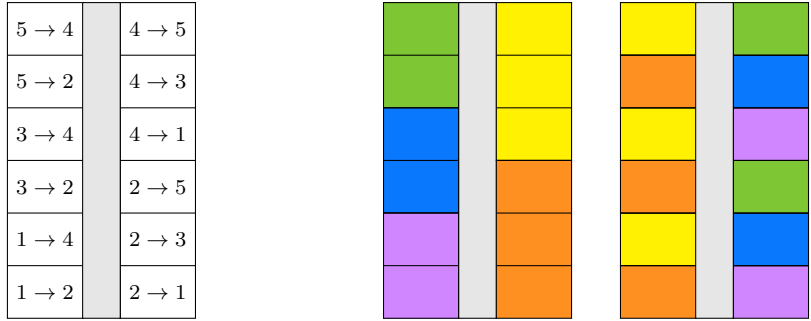
\begin{figure}[ht]
 \centering
\begin{tikzpicture}[yscale=0.7]
 \draw [fill=white] (2.5,0) rectangle (1.5,6);
 \draw [fill=white] (0,0) rectangle (1,6);

\draw [fill=black!10!white] (1.5,6) rectangle (1,0);

\foreach \x in {1,2,3,4,5}{
\draw (0,\x)--(1,\x);
\draw (1.5,\x)--(2.5,\x);
}

\node at (0.5,0.5) {\footnotesize{$1 \to 2$}};
\node at (0.5,1.5) {\footnotesize{$1 \to 4$}};
\node at (0.5,2.5) {\footnotesize{$3 \to 2$}};
\node at (0.5,3.5) {\footnotesize{$3 \to 4$}};
\node at (0.5,4.5) {\footnotesize{$5 \to 2$}};
\node at (0.5,5.5) {\footnotesize{$5 \to 4$}};

\node at (2,0.5) {\footnotesize{$2 \to 1$}};
\node at (2,1.5) {\footnotesize{$2 \to 3$}};
\node at (2,2.5) {\footnotesize{$2 \to 5$}};
\node at (2,3.5) {\footnotesize{$4 \to 1$}};
\node at (2,4.5) {\footnotesize{$4 \to 3$}};
\node at (2,5.5) {\footnotesize{$4 \to 5$}};

\begin{scope}[xshift=5cm]
 \draw [fill=green] (0,4) rectangle (1,6);
 \draw [fill=blue] (0,2) rectangle (1,4);
 \draw [fill=purple] (0,0) rectangle (1,2);
\draw [fill=black!10!white] (1.5,6) rectangle (1,0);
 \draw [fill=yellow] (2.5,3) rectangle (1.5,6);
 \draw [fill=red] (2.5,0) rectangle (1.5,3);

\foreach \x in {1,2,3,4,5}{
\draw (0,\x)--(1,\x);
\draw (1.5,\x)--(2.5,\x);
}
\end{scope}

\begin{scope}[xshift=8.15cm]
\foreach \x in {1,3,5}{
 \draw [fill=yellow] (0,\x) rectangle (1,\x+1);
 \draw [fill=red] (0,\x) rectangle (1,\x-1);
}
\foreach \x in {0,3}{
 \draw [fill=purple] (2.5,\x) rectangle (1.5,\x+1);
 \draw [fill=blue] (2.5,\x+2) rectangle (1.5,\x+1);
  \draw [fill=green] (2.5,\x+2) rectangle (1.5,\x+3);
  }
\draw [fill=black!10!white] (1.5,6) rectangle (1,0);

\foreach \x in {1,2,3,4,5}{
\draw (0,\x)--(1,\x);
\draw (1.5,\x)--(2.5,\x);
}

\end{scope}

\end{tikzpicture}
\captionsetup{width=.85\textwidth}
\caption{Initial and final $L$-colouring of $K_{18,18}$ where every ``box" represents $3$ vertices (having different lists), which requires at least $55$ recolourings. Left: These colourings are shown as $\a(v) \to \b(v)$. Right: These colourings are shown with explicit colours.}
 \label{fig:SwitchColoursK18_18}
\end{figure}
 
 Note that due to the choice of lists, we cannot colour either part with only a single colour.
 So each part must always use at least $2$ colours; so $4$ colours in total.
 This implies that we can only move one colour at a time between $U$ and $V$.
 When recolouring $\a$ into $\b$, 
 we may assume, by possibly renaming colours, that 1 is the first colour completely removed from its part (which must be $U$), followed next by 2 (removed from $V$), and followed next by 3 (removed from $U$).
 (Note that completely removing a colour from a part and then immediately again using that colour on that part will always be inefficient.) 
 While removing colours 1, 2, 3 from their parts, we need to recolour some vertices without giving each of them its final colour; we call each instance of this a \emph{bad} recolouring.
 We will prove that, to recolour $\a$ to $\b$, we need at least $19$ bad recolourings. This will finish the proof, since then the total number of recolourings we use to get from $\a$ to $\b$ is at least $19+|U|+|V|=19+36=55$, as desired.
 
 We start by recolouring vertices (in $U$) initially coloured $1$ with either $3$ or $5$;
 we do this for at least 6 vertices, so begin with at least 6 bad recolourings. 
 Next we remove colour $2$ from $V$. Initially, $2$ is used on 9 vertices of $V$; since 6 of these are coloured with $3$ or $5$ in $\b$, we perform at least 6 more bad recolourings.
 Initially, 6 vertices of $U$ are coloured with $3$. In $\b$, we use colour $4$ on exactly 3 of these; but colour $4$ is not yet available for use on $U$, so we need 3 more bad recolourings. In addition,
 consider the vertex $u\in U$ that corresponds to triple $(1,4,5)$; that is $\a(u)=1$, $\b(u)=4$, and $L(u)=[5]\setminus\{5\}=[4]$. When $1$ was removed from $U$, vertex $u$ must have been recoloured with $3$. So now that colour $3$ is
 being removed, we need another bad recolouring.

 When we next completely remove a colour from 
 $V$, it is either $1$ or $4$. If it $4$, then 
 we need at least 3 more bad recolourings, since
 colour $5$ is not yet available to use on $V$.
 But if we instead remove $1$, then we will also
 incur at least 3 more bad recolourings (the vertices in $V$ with $\a(v)=2, \b(v)=1$ all have been recoloured badly as well).
 Thus, the total number of bad recolourings is at least $6+6+3+1+3 = 19$, as desired.
\end{proof}

The following result, combined with \cref{extension:rem},
extends~\cref{prop:C_k&C_Lunrelated}(d), exhibiting an even larger (unbounded) gap between $\diam \C_L(G)$ and $\diam \C_k(G)$.

\begin{lem}
For all $n\geq 1$ there exists an $n$-vertex graph $G_n$ with a $4$-fold list-assignment $L$, such that $c_1 \cdot n \leq \frac{\diam \C_L (G_n)}{ \diam \C_4(G_n)} \leq c_2 \cdot n$, where $c_1$ and $c_2$ are some uniform positive constants. 
\label{path-plus-lem}
\end{lem}
\begin{proof}
For each positive integer $k$, we construct a graph $G_k$
with $|V(G_k)|=260k$ and a $4$-assignment $L$ such that
$(k^2-1)/4+255k\le \diam \C_L(G_k)\le 2k^2+255k$ and $260k\le \diam \C_4(G_k)\le 691k$. We use these graphs $G_k$ to prove the result when $n=260k$ for some positive integer $k$. It is straightforward (but a bit tedious) to extend the result to all positive integers $n$; so we just sketch the ideas. 
In general when $n\ge 260$, we pick an integer $k$ such that $260k\le n < 260 (k+1)$, and we take our graph $G$ to be an $n$-vertex subgraph of $G_{k+1}$ that contains as an induced subgraph $G_k$. Now $n/2\le 260k \le n < 260 (k+1) \le 2n$, so by applying the bounds for $G_k$ and $G_{k+1}$, we lose at most a constant factor in both the upper and lower bounds.

The most interesting part of our construction is the lower bound on $\diam\C_L(G_k)$, so we begin with that. We start with a path of order $k$ and ``hang off'' of each path vertex a copy of $K_{4,255}$ that uses its lists to enforce that the
recolouring of the path simulates the recolouring of
a path with only 3 colours (and no lists).
Since~\cite{BJLPP14} showed that $(k^2-1)/4\le \diam \C_3(P_k)\le 2k^2$, we are done. Thus, we need only to describe these
lists and show how they enforce the desired simulation
of $\C_3(P_k)$.

Let $H$ be the complete bipartite graph $K_{4,255}$. To the four vertices $v_1,v_2,v_3,v_4$ in the part of size $4$, assign four disjoint lists: $L(v_1)=\{1,2,3,4\},L(v_2)=\{5,6,7,8\},L(v_3)=\{9,10,11,12\},L(v_4)=\{13,14,15,16\}$. In the part of size $255=4^4-1$, assign each vertex a unique list that contains precisely one element from each of $L(v_1),L(v_2),L(v_3),L(v_4)$, except that no vertex is assigned $\{1,5,9,13\}$. Note that in every proper $L$-colouring $c$ of $H$, we must have $c(v_1)=1,c(v_2)=5,c(v_3)=9$ and $c(v_4)=13$. From this it follows that $\C_L(H)$ is connected, as any two $L$-colourings only differ on the part of size $255$ (since that is an independent set, in fact the diameter of $\C_L(H)$ is $255$). In what follows we call $v_1$ the \emph{special vertex} of $H$ and we only need that it is coloured $1$ in every $L$-colouring. 

Now consider a path $P_k=p_1,\ldots, p_k$ on $k$ vertices. Let $G_k$ denote the graph formed by taking the disjoint union of $P_k$ and $k$ copies $H_1,\ldots, H_k$ of $H$, and then for every $i\in[k]$ adding an edge between $p_i$ and the special vertex of $H_i$.

Here we tacitly assumed that each $H_i$ is equipped with the list-assignment $L$ described above. Moreover, we assign each vertex of $P$ the list $\{1,a,b,c\}$, where $\{a,b,c\}$ is disjoint from $\{1,2,\ldots, 16\}$. Since in a proper $L$-colouring each special vertex must have colour $1$, it follows that each vertex of $P_k$ must use a colour in $\{a,b,c\}$. In other words, restricted to $P$ the $L$-colourings correspond precisely to the $3$-colourings of $P$.
As noted above, by~\cite[Thm.~3]{BJLPP14}, we know that $(k^2-1)/4\le \diam \C_3(P_k)\le 2k^2$. 

It follows that $\diam \C_L(G_k) = \diam \C_L(P_k) + k \cdot\diam \C_L(H) = \diam \C_3(P_k) + 255 k= \Theta(k^2)$, because any two $L$-colourings of $G_k$ only differ on $P_k$ and the size-$255$ parts of the copies of $H$, taking note that these parts are separated from each other by the special vertices that have the same colour in every $L$-colouring. \\

Now we must prove that $\diam \C_4(G_k) = \Theta(k)$. 
The lower bound is trivial, since $|V(G_k)|=(255+4+1)k$ and permuting the colors (with no fixed points) in an arbitrary colouring $\a$ gives a colouring $\b$ with $d(\a,\b)\ge |V(G_k)|=260k$.
For the upper bound, let $\alpha, \beta$ be two $4$-colourings of $G_k$, and let $\alpha(P_k), \beta(P_k)$ be their restriction to $P_k$. As $4=\Delta(P_k)+2$, it follows (e.g. see \cite[Thm.~1]{CCC24}) that $\diam \C_4(P_k) \leq 2k$. Thus there is a sequence $\mathcal{S}$ of at most $2k$ vertex recolourings that transforms $\alpha(P_k)$ into $\beta(P_k)$ while maintaining a proper colouring of $P_k$, but not necessarily of $G_k$. 
Consider an arbitrary step of $\mathcal{S}$ where we need to recolour some vertex $p_i$ of $P_k$, without loss of generality from colour $1$ to colour $2$. This recolouring step is also allowed in $G_k$ unless the special vertex $v(H_i)$ of $H_i$ has colour $2$ as well. In this case, we first recolour $v(H_i)$ to either $3$ or $4$. 
This can be done in one step if either $3$ or $4$ does not appear in the size-$255$ part $B_i$ of $H_i$. Otherwise, we recolour $B_i$ to remove either $3$ or $4$ from $B_i$, using at most $\lfloor 255/2\rfloor=127$ steps, and then recolour $v(H_i)$ to either $3$ or $4$. In total we need at most $128$ extra recolourings to accomodate the recolouring of $p_i$. 
Doing this for every step of $\mathcal{S}$, we need at most $129\cdot 2k$ steps to recolour $\alpha$ in such a way that it agrees with $\beta$ on the vertices of $P_k$. 

After recolouring $P_k$, we must finally recolour each copy of $H$, which can be done in at most $433$ 
steps per copy. (This needs a small argument; it does not follow directly from~\cref{thr:main_diamCk(Kpq)} because throughout the colour $\beta(p_i)$ is forbidden at the special vertex of $H_i$.) Here is one way to prove this upper bound. 
Let $\widehat{\a}$ denote the current colouring. For some colour $c\in L(v_1)$, we first remove $c$ from part $B_i$ (using $\b(w)$ for each $w\in B_i$ when possible, and using some other colour in $\b(B_i)$ when needed); then we recolour $\{v_1,v_2,v_3,v_4\}$ with $c$; then we recolour each $w\in B_i\setminus \b^{-1}(c)$ with $\b(w)$; then we recolour each $v_i$ with $\b(v_i)$; and finally we recolour each $w\in \b^{-1}(c)$ with $c$. We choose colour $c$ to minimise $|\widehat{\a}^{-1}(c)\cap V|+|\b^{-1}(c)\cap V|$. Thus, the number of recolouring steps is at most $|\widehat{\a}^{-1}(c)\cap V|+4+|B_i|+4+|\b^{-1}(c)\cap V|=255+8+|\widehat{\a}^{-1}(c)\cap V|+|\b^{-1}(c)\cap V|\le 263+\lfloor 255\cdot 2/3\rfloor =263+170=433$.
Thus, in total the distance between $\alpha$ and $\beta$ is at most $129\cdot 2k + 433k = 691k$.
\end{proof}

\begin{remark}
\label{extension:rem}
 The construction above can be generalized to $k$-colourings and $k$-fold list-assignments for any $k\geq 4$, by replacing the graph $H$ with a larger complete bipartite graph, and replacing the path with a $(k-2)$-colourable $n$-vertex graph with $(k-1)$-colour
diameter $\Theta(n^2)$ and $k$-colour diameter $\Theta(n)$, the existence of which is guaranteed by constructions in~\cite{BJLPP14}.
\end{remark} 

\paragraph{Open access statement.} For the purpose of open access,
a CC BY public copyright license is applied
to any Author Accepted Manuscript (AAM)
arising from this submission.

\bibliographystyle{habbrv}
{\small
\bibliography{reconfiguration}

\begin{thebibliography}{10}
\expandafter\ifx\csname url\endcsname\relax
  \def\url#1{\texttt{#1}}\fi
\expandafter\ifx\csname doi\endcsname\relax
  \def\doi#1{\burlalt{doi:#1}{http://dx.doi.org/#1}}\fi
\expandafter\ifx\csname urlprefix\endcsname\relax\def\urlprefix{ }\fi
\expandafter\ifx\csname href\endcsname\relax
  \def\href#1#2{#2}\fi
\expandafter\ifx\csname burlalt\endcsname\relax
  \def\burlalt#1#2{\href{#2}{#1}}\fi

\bibitem{Bard2014}
S.~Bard.
\newblock {\em Gray code numbers of complete multipartite graphs}.
\newblock PhD thesis, 2014.

\bibitem{BC25}
M.~Belavadi and K.~Cameron.
\newblock Recoloring some hereditary graph classes.
\newblock {\em Discrete Applied Mathematics}, 361:389--401, 2025, \burlalt{\tt
  arXiv:2312.00979}{http://arxiv.org/abs/2312.00979}.
\newblock \doi{10.1016/j.dam.2024.10.026}.

\bibitem{BCH24}
M.~Belavadi, K.~Cameron, and E.~Hildred.
\newblock Frozen colourings in {$2K_2$}-free graphs.
\newblock 2024, \burlalt{\tt
  arXiv:2409.13161}{http://arxiv.org/abs/2409.13161}.

\bibitem{BCM24}
M.~Belavadi, K.~Cameron, and O.~Merkel.
\newblock Reconfiguration of vertex colouring and forbidden induced subgraphs.
\newblock {\em European J. Combin.}, 118:Paper No. 103908, 10, 2024,
  \burlalt{\tt arXiv:2206.09268}{http://arxiv.org/abs/2206.09268}.
\newblock \doi{10.1016/j.ejc.2023.103908}.

\bibitem{BB18}
M.~Bonamy and N.~Bousquet.
\newblock Recoloring graphs via tree decompositions.
\newblock {\em European J. Combin.}, 69:200--213, 2018, \burlalt{\tt
  arXiv:1403.6386}{http://arxiv.org/abs/1403.6386}.
\newblock \doi{10.1016/j.ejc.2017.10.010}.

\bibitem{BJLPP14}
M.~Bonamy, M.~Johnson, I.~Lignos, V.~Patel, and D.~Paulusma.
\newblock Reconfiguration graphs for vertex colourings of chordal and chordal
  bipartite graphs.
\newblock {\em J. Comb. Optim.}, 27(1):132--143, 2014.
\newblock \doi{10.1007/s10878-012-9490-y}.

\bibitem{BFHR22}
N.~Bousquet, L.~Feuilloley, M.~Heinrich, and M.~Rabie.
\newblock Short and local transformations between ({$\Delta$}+1)-colorings.
\newblock 2022, \burlalt{\tt
  arXiv:2203.08885}{http://arxiv.org/abs/2203.08885}.

\bibitem{BP15}
N.~Bousquet and G.~Perarnau.
\newblock Fast recoloring of sparse graphs.
\newblock {\em European J. Combin.}, 52(part A):1--11, 2016, \burlalt{\tt
  arXiv:1411.6997}{http://arxiv.org/abs/1411.6997}.
\newblock \doi{10.1016/j.ejc.2015.08.001}.

\bibitem{CCC24}
S.~Cambie, W.~Cames~van Batenburg, and D.~W. Cranston.
\newblock Optimally reconfiguring list and correspondence colourings.
\newblock {\em Eur. J. Comb.}, 115:20, 2024, \burlalt{\tt
  arXiv:2204.07928}{http://arxiv.org/abs/2204.07928}.
\newblock \doi{10.1016/j.ejc.2023.103798}.
\newblock Id/No 103798.

\bibitem{CvdHJ08}
L.~Cereceda, J.~van~den Heuvel, and M.~Johnson.
\newblock Connectedness of the graph of vertex-colourings.
\newblock {\em Discrete Math.}, 308(5-6):913--919, 2008.
\newblock \doi{10.1016/j.disc.2007.07.028}.

\bibitem{cranston22}
D.~W. Cranston.
\newblock List-recoloring of sparse graphs.
\newblock {\em European J. Combin.}, 105:Paper No. 103562, 11, 2022,
  \burlalt{\tt arXiv:2201.05133}{http://arxiv.org/abs/2201.05133}.
\newblock \doi{10.1016/j.ejc.2022.103562}.

\bibitem{DF20}
Z.~Dvo\v{r}\'{a}k and C.~Feghali.
\newblock An update on reconfiguring 10-colorings of planar graphs.
\newblock {\em Electron. J. Combin.}, 27(4):Paper No. 4.51, 21, 2020,
  \burlalt{\tt arXiv:2002.05383}{http://arxiv.org/abs/2002.05383}.
\newblock \doi{10.37236/9391}.

\bibitem{DF21}
Z.~Dvo\v{r}\'{a}k and C.~Feghali.
\newblock A {T}homassen-type method for planar graph recoloring.
\newblock {\em European J. Combin.}, 95:Paper No. 103319, 12, 2021,
  \burlalt{\tt arXiv:2006.09269}{http://arxiv.org/abs/2006.09269}.
\newblock \doi{10.1016/j.ejc.2021.103319}.

\bibitem{jerrum}
M.~Jerrum.
\newblock A very simple algorithm for estimating the number of {$k$}-colorings
  of a low-degree graph.
\newblock {\em Random Structures Algorithms}, 7(2):157--165, 1995.
\newblock \doi{10.1002/rsa.3240070205}.

\bibitem{merkel}
O.~Merkel.
\newblock Recolouring weakly chordal graphs and the complement of triangle-free
  graphs.
\newblock {\em Discrete Math.}, 345(3):Paper No. 112708, 5, 2022, \burlalt{\tt
  arXiv:2106.11087}{http://arxiv.org/abs/2106.11087}.
\newblock \doi{10.1016/j.disc.2021.112708}.

\bibitem{Mutze23}
T.~M\"utze.
\newblock Combinatorial {G}ray codes---an updated survey.
\newblock {\em Electron. J. Combin.}, DS26:93, 2023.
\newblock \doi{10.37236/11023}.

\end{thebibliography}
}

\end{document}